\documentclass[a4paper,12pt]{article}

\title{The structure of groups with\\
  all proper quotients virtually nilpotent}
\author{Benjamin Klopsch \& Martyn Quick}

\usepackage{amsmath,amssymb}
\usepackage[margin=3cm]{geometry}
\usepackage{mathrsfs}
\usepackage[only,llbracket,rrbracket,trianglelefteqslant]{stmaryrd}
\usepackage{ntheorem}
\usepackage[noadjust]{cite}
\usepackage{enumitem}

\usepackage{xcolor}
\usepackage{hyperref}
\hypersetup{colorlinks=true, linkcolor={red!50!black}, citecolor={green!50!black}}

\theorembodyfont{\slshape}
\newtheorem{thm}{Theorem}[section]
\newtheorem{lemma}[thm]{Lemma}
\newtheorem{prop}[thm]{Proposition}
\newtheorem{cor}[thm]{Corollary}
\newtheorem{THM}{Theorem}
\theorembodyfont{\normalfont}
\newtheorem{defn}[thm]{Definition}
\newtheorem{example}[thm]{Example}

\newcommand{\1}{\mathbf{1}}
\newcommand{\Aset}[1]{\mathcal{A}_{#1}}
\newcommand{\Astarset}[1]{\mathcal{A}_{#1}^{\ast}}
\newcommand{\Aut}[1]{\operatorname{Aut}#1}
\newcommand{\Autc}[1]{\operatorname{Aut}_{\mathrm{c}}#1}
\newcommand{\Cent}[2]{\mathrm{C}_{#1}(#2)}
\newcommand{\Centre}[2][]{\mathrm{Z}_{#1}(#2)}
\newcommand{\Core}[2]{\mathrm{Core}_{#1}(#2)}
\newcommand{\Cset}[1]{\mathcal{C}_{#1}}
\newcommand{\Cstarset}[1]{\mathcal{C}_{#1}^{\ast}}
\newcommand{\dirlim}{\varinjlim}
\newcommand{\Field}[1]{\mathbb{F}_{#1}}
\newcommand{\Fin}[1]{\mathrm{Fin}(#1)}
\newcommand{\Fitt}[1]{\mathrm{F}(#1)}
\newcommand{\GL}[1]{\mathrm{GL}(V)}
\newcommand{\group}[2]{\langle\,#1\mid#2\,\rangle}

\newcommand{\Inn}[1]{\operatorname{Inn}#1}
\newcommand{\invlim}{\varprojlim}
\newcommand{\JNNcF}{JNN$_{c}$F}
\newcommand{\lcm}{\operatorname{lcm}}
\newcommand{\leqc}{\leq_{\mathrm{c}}}
\newcommand{\leqo}{\leq_{\mathrm{o}}}
\newcommand{\lowerc}[2]{\gamma_{#1}(#2)}
\newcommand{\Mel}[2][]{\mathrm{M}_{#1}(#2)}
\newcommand{\mult}{^{\ast}}
\newcommand{\Nat}{\mathbb{N}}
\newcommand{\nbd}{\nobreakdash-}
\newcommand{\Norm}[2]{\mathrm{N}_{#1}(#2)}
\newcommand{\normal}{\trianglelefteqslant}
\newcommand{\normalc}{\normal_{\mathrm{c}}}
\newcommand{\normalo}{\normal_{\mathrm{o}}}
\newcommand{\Nott}{\mathcal{N}}
\newcommand{\order}[1]{\mathopen{|}#1\mathclose{|}}
\renewcommand{\P}{\mathscr{P}}
\newcommand{\PSL}[2]{\mathrm{PSL}_{#1}(#2)}
\newcommand{\psring}{\Field{p}\llbracket T\rrbracket}
\newcommand{\set}[2]{\{\,#1\mid#2\,\}}
\newcommand{\SL}[2]{\mathrm{SL}_{#1}(#2)}
\newcommand{\slLie}[2]{\mathfrak{sl}_{#1}(#2)}
\newcommand{\SLc}[2]{\mathrm{SL}_{#1}^{1}(#2)}
\newcommand{\succnar}{\succ_{\text{nar}}}
\newcommand{\Zint}{\mathbb{Z}}

\newcommand{\AND}{\qquad\text{and}\qquad}

\newcommand{\spc}{\vspace{\baselineskip}}

\renewcommand{\leq}{\leqslant}
\renewcommand{\geq}{\geqslant}
\renewcommand{\nleq}{\nleqslant}

\renewcommand{\wr}{\operatorname{wr}}

\newenvironment{proof}{%
  \begin{trivlist}\item[]\textsc{Proof:}}{%
  \qed\end{trivlist}}
\newcommand{\qed}{\hspace*{\fill}$\square$}

\begin{document}

\maketitle

\begin{abstract}
  Just infinite groups play a significant role in profinite group
  theory.  For each $c \geq 0$, we consider more generally
  \JNNcF\ profinite (or, in places, discrete) groups that are
  Fitting-free; these are the groups~$G$ such that every proper
  quotient of $G$ is virtually class\nbd$c$ nilpotent whereas $G$
  itself is not, and additionally $G$ does not have any non-trivial
  abelian normal subgroup.  When $c = 1$, we obtain the just
  non-(virtually abelian) groups without non-trivial abelian normal
  subgroups.

  Our first result is that a finitely generated profinite group is
  virtually class\nbd$c$ nilpotent if and only if there are only
  finitely many subgroups arising as the lower central series
  terms~$\lowerc{c+1}{K}$ of open normal subgroups $K$ of~$G$.  Based
  on this we prove several structure theorems.  For instance, we
  characterize the \JNNcF\ profinite groups in terms of subgroups of
  the above form~$\lowerc{c+1}{K}$.  We also give a description of
  \JNNcF\ profinite groups as suitable inverse limits of virtually
  nilpotent profinite groups.  Analogous results are established for
  the family of hereditarily \JNNcF\ groups and, for instance, we show
  that a Fitting-free \JNNcF\ profinite (or discrete) group is
  hereditarily \JNNcF\ if and only if every maximal subgroup of finite
  index is \JNNcF\@.  Finally, we give a construction of hereditarily
  \JNNcF\ groups, which uses as an input known families of
  hereditarily just infinite groups.
\end{abstract}


\section{Introduction and Main Results}

If $\P$~is a property of groups, a group~$G$ is said to be \emph{just
  non\nbd$\P$} when $G$~does not have property~$\P$ but all proper
quotients of~$G$ do satisfy~$\P$.  In the case when $G$~is a profinite
group, we require instead that every quotient of~$G$ by a non-trivial
\emph{closed} normal subgroup has~$\P$.  The property~$\P$ considered
most often has been that of being finite and the more common term
\emph{just infinite} is then used.  Just infinite groups are
particularly important within the context of profinite -- or more
generally residually finite -- groups, since infinite residually
finite groups are never simple but instead just infinite groups can be
viewed as those with all proper quotients \emph{essentially} trivial
from a `residually finite' viewpoint (see, for example, the discussion
in~\cite[\S12.1]{LGM}).  Important examples of just infinite groups
include the Grigorchuk group~\cite{Grig84} and the Nottingham
group~\cite{Klopsch-Nott,Hegedus}, but also families arising as
quotients of arithmetic groups by their centres~\cite{BMS}.

There is a dichotomy in the study of just non\nbd$\P$ groups.  One
thread within the literature is concerned with the study of just
non\nbd$\P$ groups possessing a non-trivial normal abelian subgroup.
In this context, a key idea is to exploit the structure of a maximal
abelian normal subgroup when viewed as a module in the appropriate
way.  Studies of this type
include~\cite{McCarthy1,McCarthy2,deFalco,MRQ} and we also refer to
the monograph~\cite{KOS} for more examples.  On the other hand,
Wilson~\cite{JSW71,JSW00} addresses the case of just infinite groups
with no non-trivial abelian normal subgroup.  He shows that such
groups fall into two classes: (i)~branch groups, and (ii)~certain
subgroups of wreath products of a hereditarily just infinite group by
a symmetric group of finite degree.  The class of branch groups has
been studied considerably (see, for example, \cite{BGS,Grig00}, though
many more articles on these groups have appeared since these surveys
were written).  It is known that every proper quotient of a branch
group is virtually abelian (see the proof of~\cite[Theorem~4]{Grig00})
and there are examples of branch groups that are not just infinite
(see~\cite{Fink}, for example).  It is interesting therefore to note
that Wilson's methods extend to the class of groups with all proper
quotients virtually abelian, as observed by Hardy in his PhD
thesis~\cite{Hardy-PhD}.  We shall use the abbreviation \emph{JNAF
groups} for these just non-(abelian-by-finite) groups.

More recently, Reid established various fundamental results concerning
the structure and properties of just infinite groups
(see~\cite{Reid10a,Reid10b,Reid12,Reid-corr}).  One might wonder to
what extent JNAF groups have a similar structure to just infinite
groups.  In this article, we demonstrate how, for fixed~$c \geq 0$,
Reid's results may in fact be extended to the even larger class of
groups with all proper quotients being virtually nilpotent of class at
most~$c$; that is, the just non-(class\nbd$c$-nilpotent-by-finite)
groups.  We shall abbreviate this term to \emph{\JNNcF\ group} in what
follows.  The case $c = 0$ essentially returns Reid's results, while
the case $c = 1$ covers all JNAF groups and so, in particular, would
apply to all branch groups.

We do require some additional, though rather mild, hypotheses to those
appearing in Reid's work.  First, the \JNNcF\ groups that we consider
will be assumed to be \emph{Fitting-free}; that is, to have no
non-trivial abelian normal subgroup.  This is consistent with Wilson's
and Hardy's studies and with the viewpoint that says that the case
with a non-trivial abelian normal subgroup should be studied through a
module-theoretic lens.  (As an aside, we emphasize that \JNNcF\ groups
with some non-trivial abelian normal subgroup are, in particular,
abelian-by-nilpotent-by-finite.)  Infinite Fitting-free groups cannot
be virtually nilpotent, so part of the definition of \JNNcF\ group
comes immediately.  In addition, we shall frequently assume that the
groups under consideration are finitely generated.  This latter
condition will enable us to control the structure of the quotients
that arise.

It is interesting to note which parts of Reid's ideas adapt readily to
the \JNNcF\ setting and where differences occur.  One example is that
he implicitly uses the fact that a proper quotient of a just infinite
group, being finite, has only finitely many subgroups.  In contrast,
any infinite (virtually nilpotent) quotient of a profinite group will
necessarily have infinitely many open normal subgroups.  We shall
depend upon the following result as a key tool in our work.  It means
that, while a finitely generated virtually nilpotent profinite group
typically has infinitely many closed normal subgroups, it only has
finitely many that occur as corresponding lower central series
subgroups of open normal subgroups.

\begin{THM}
  \label{thm:maintool}
  Let $G$~be a finitely generated profinite group.  Then $G$~is
  virtually nilpotent of class at most~$c$ if and only if the set
  $\set{ \lowerc{c+1}{K} }{K \normalo G}$ is finite.
\end{THM}

The above result is established as Theorem~\ref{thm:Cset-finite} in
Section~\ref{sec:prelims}.  In that section, we also give precise
definitions and recall properties needed during the course of our
work.

In Section~\ref{sec:JNNF}, we fix an integer~$c \geq 0$ and
investigate the structure of \JNNcF\ profinite groups $G$ that are
Fitting-free.  We shall establish various descriptions that generalize
those of just infinite groups in~\cite{Reid10b,Reid12,Reid-corr}.  One
point that can be noted is that the subgroups of the
form~$\lowerc{c+1}{K}$, for $K$~an open normal subgroup of~$G$, play a
role in \JNNcF\ groups analogous to that of open normal subgroups in
just infinite groups.  For example, we show that a directed
graph~$\Gamma$ can be constructed from a suitable subcollection of
$\set{ \lowerc{c+1}{K} }{ K \normalc G }$ that is locally finite.
This enables us to establish our first characterization of
\JNNcF\ groups (established as Theorem~\ref{thm:Aset} below), which is
the following analogue of Reid's ``Generalized Obliquity Theorem''
\cite[Theorem~A]{Reid10b}.  Specifying $c = 0$ results in a mild
weakening of Reid's theorem.

\begin{THM}
  \label{thm:obliq}
  Let $G$~be a finitely generated infinite profinite group that has no
  non-trivial abelian closed normal subgroup.  Then $G$~is \JNNcF\ if
  and only if the set $\Aset{H} = \set{ \lowerc{c+1}{K} }{
    \text{$K \normalo G$ with $\lowerc{c+1}{K} \nleq H$} }$ is finite
  for every open subgroup~$H$ of~$G$.
\end{THM}

This result is used to characterize, in Theorem~\ref{thm:JNNFstruct},
when a finitely generated Fitting-free profinite group is \JNNcF.  The
characterization is expressed in terms of properties of a descending
chain of open normal subgroups~$H_{i}$ and the $(c+1)$-th term of
their lower central series.  In Theorem~\ref{thm:JNNF-invlim}, we
establish a further characterization of such a group as an inverse
limit in a manner analogous to \cite[Theorem~4.1]{Reid12}.  One
important tool (see Lemma~\ref{lem:JNNF-Mel}) that is used throughout
Section~\ref{sec:JNNF} is that, if $G$~is a Fitting-free
\JNNcF\ profinite group and $N$~is a non-trivial closed normal
subgroup, then the Melnikov subgroup~$\Mel{N}$ of~$N$ is non-trivial
and so, via the Fitting-free assumption, $\lowerc{i}{\Mel{N}} \neq \1$
for all~$i \geq 1$.

Section~\ref{sec:hJNNF} is concerned with the structure of profinite
groups that are hereditarily \JNNcF\@.  We establish there a similar
suite of results, though the description of a finitely generated,
Fitting-free hereditarily \JNNcF\ group as an inverse limit is more
technical (see Theorem~\ref{thm:Hinvlim}).  It shares this level of
technicality with Reid's characterization of hereditarily just
infinite groups.

In Section~\ref{sec:finiteindex}, we establish the following (as
Corollary~\ref{cor:max}) which is the analogue of the main result
of~\cite{Reid10b}.  The material in this section does not depend upon
Theorem~\ref{thm:maintool} and so is more directly developed from
Reid's arguments.

\begin{THM}
  \label{thm:Hmax}
  Let $G$~be a \JNNcF\ profinite or discrete group that has no
  non-trivial abelian normal subgroup.  Then $G$~is hereditarily
  \JNNcF\ if and only if every maximal (open) subgroup of finite index
  is \JNNcF.
\end{THM}

One reasonable conclusion from the results described so far is that
there is a similarity in the structure of \JNNcF\ groups when compared
to just infinite groups.  One might ask: just how closely are these
classes linked?  As \JNNcF\ groups have not yet been studied
systematically, there are presently rather few examples to examine
when considering these links.  In the final section of the paper,
Section~\ref{sec:construction}, we take a first step and present one
way to construct hereditarily \JNNcF\ groups from hereditarily just
infinite groups as semidirect products and discuss some explicit
examples.  We give examples of hereditarily JNAF groups of the form~$G
\rtimes A$ where $G$~can be a hereditarily just infinite group
suitably built as an iterated wreath product or using Wilson's
Construction~B from~\cite{JSW-Large} and $A$~can be selected from a
rather broad range of abelian groups (see Examples~\ref{ex:JNAF1}
and~\ref{ex:JNAF2}).  By exploiting the fact that every countable
pro\nbd$p$ group can be embedded in the Nottingham group, we construct
a hereditarily \JNNcF\ group of the form $\SLc{n}{\psring} \rtimes A$
where $A$~can be any virtually nilpotent pro\nbd$p$ group (see
Example~\ref{ex:Nott}).  This last family of examples demonstrates
that, for every possible choice of~$c \geq 1$, there is a
\JNNcF\ pro-$p$ group that is not just non-(virtually nilpotent of
smaller class).

Since the examples constructed are built using hereditarily just
infinite groups, one is drawn back to the above question concerning
the link between \JNNcF\ groups and just infinite groups.  The results
of Sections~\ref{sec:JNNF}--\ref{sec:finiteindex} suggest such a link
and it is an open challenge to produce examples of hereditarily
\JNNcF\ groups of a compellingly different flavour to those built in
Section~\ref{sec:construction}.


\section{Preliminaries}
\label{sec:prelims}

In this section, we first give the precise definitions of the groups
under consideration.  We then recall some useful tools
from~\cite{Reid12} and make a number of basic observations about
\JNNcF\ groups.  In the last part of the section we consider the
behaviour of finitely generated virtually nilpotent groups and
establish Theorem~\ref{thm:maintool} which is crucial for the sections
that follow.

We shall write maps on the right throughout, so $H\phi$~denotes the
image of a group~$H$ under a homomorphism~$\phi$ and $x^{y}$~is the
conjugate~$y^{-1}xy$.  If $G$~is a profinite group, we use the usual
notation $H \normalo G$ and $K \normalc G$ for an open normal subgroup
and a closed normal subgroup, respectively.  If $K$~and~$L$ are closed
subgroups of~$G$, then $[K,L]$~will denote the \emph{closed} subgroup
generated by all commutators $[x,y] = x^{-1}y^{-1}xy$ where $x \in K$
and $y \in L$.  The \emph{lower central series} of~$G$ is then defined
by $\lowerc{1}{G} = G$ and $\lowerc{i+1}{G} = [\lowerc{i}{G},G]$ for
each~$i \geq 1$.  As usual, we also use~$G'$ for the derived
subgroup~$\lowerc{2}{G}$ of~$G$.  These concepts will, in particular,
be relevant for the instances of the following definition that concern
us.

\begin{defn}
  \label{def:justnot}
  Let $\P$~be a property of groups.  A profinite (or discrete)
  group~$G$ is said to be \emph{just non\nbd$\P$} if $G$~does not have
  property~$\P$ but $G/N$~does have~$\P$ for every non-trivial closed
  normal subgroup~$N$ of~$G$.  It is \emph{hereditarily just
  non\nbd$\P$} if every closed subgroup of finite index in~$G$ is just
  non\nbd$\P$.

  (When $G$~is discrete, the word ``closed'' can and should be
  ignored.  Note that a closed subgroup of finite index is necessarily
  open, but the definition is phrased to enable that for discrete
  groups to be readily extracted.)
\end{defn}

In this paper we shall consider three options for the property~$\P$:
\begin{enumerate}[label=(\arabic*)]
\item When $\P$~is the property of being finite, we use the more
  common term \emph{just infinite} for an infinite group with every
  proper quotient finite.
\item We use the abbreviation \emph{JNAF} for just non\nbd$\P$ when
  $\P$~is the property of being virtually abelian, which is the same
  as being abelian-by-finite.  A profinite group has an abelian
  subgroup of finite index if and only if it has an abelian open
  subgroup (as the topological closure of any abelian subgroup is
  again abelian), so we use the term virtually abelian in this
  situation also.
\item If $c$~is an integer with $c \geq 0$, we use the abbreviation
  \emph{\JNNcF} for just non\nbd$\P$ when $\P$~is the property that
  there is a subgroup~$H$ of finite index such that
  $\lowerc{c+1}{H} = \1$.  A profinite group has a class\nbd$c$
  nilpotent subgroup of finite index if and only if it has an open
  class\nbd$c$ nilpotent subgroup (as the topological closure of any
  class\nbd$c$ nilpotent subgroup is again class\nbd$c$ nilpotent).
\end{enumerate}
The case $c = 1$ for a \JNNcF\ group is then identical to it being
JNAF\@.  We shall speak of a group~$G$ being \emph{virtually
class\nbd$c$ nilpotent} when it has a subgroup~$H$ of finite index
satisfying $\lowerc{c+1}{H} = \1$.  More precisely such a group is
``virtually (nilpotent of class at most~$c$)''.  The
\JNNcF\ groups~$G$ considered will usually be assumed to not have a
non-trivial abelian closed normal subgroup.  Consequently, such~$G$
will itself not be virtually nilpotent (of any class) and so we are
studying groups that are just non-(virtually nilpotent) with an
additional bound upon the nilpotency class occurring in the proper
quotients.  In particular, when $c = 1$ we are considering groups that
are not virtually metabelian but where every proper quotient is
virtually abelian.

Let $G$~be a profinite group.  In line
with~\cite[Definition~2.1]{Reid12}, a \emph{chief factor} of~$G$ is a
quotient~$K/L$ where $K$~and~$L$ are closed normal subgroups of~$G$
such that there is no closed normal subgroup~$M$ of~$G$ with $L < M <
K$.  Accordingly, we do not require that $K$~be open in~$G$ in this
definition, though necessarily $L$~is open in~$K$ and hence $K/L$~is
isomorphic (under an isomorphism that commutes with the action of~$G$)
to a chief factor~$K_{0}/L_{0}$ with $K_{0}$~an open normal subgroup
of~$G$.

The \emph{Melnikov subgroup}~$\Mel{G}$ of~$G$ is the intersection of
all maximal open normal subgroups of~$G$.  Provided $G$~is
non-trivial, this is a topologically characteristic proper closed
subgroup of~$G$.  As usual, to say a subgroup of~$G$ is
\emph{topologically characteristic} means that it is invariant under
all automorphisms of~$G$ that are also homeomorphisms.  We follow
\cite[Definition~3.1]{Reid12} and, for a non-trivial closed normal
subgroup~$A$ of~$G$, define $\Mel[G]{A}$~to be the intersection of all
maximal $G$\nbd invariant open subgroups of~$A$.  This satisfies
$\Mel{A} \leq \Mel[G]{A} < A$.  We call~$A$ a \emph{narrow subgroup}
of~$G$ if $A$~has a unique maximal $G$\nbd invariant open subgroup
(that is, when $\Mel[G]{A}$~is this unique subgroup).  The first part
of the following lemma is a consequence of the Correspondence Theorem,
while the other two are, respectively, Lemmas~3.2 and~3.3 in Reid's
paper~\cite{Reid12}.

\begin{lemma}
  \label{lem:Mel}
  Let $G$~be a profinite group.
  \begin{enumerate}
  \item \label{i:Mel-corr}
    Let $K$~and~$L$ be closed normal subgroups of~$G$ such that $L
    \leq \Mel[G]{K}$.  Then $\Mel[G/L]{K/L} = \Mel[G]{K}/L$.
  \item \label{i:Mel-incl}
    Let $K$~and~$L$ be closed normal subgroups of~$G$.  Then $K \leq
    L\,\Mel[G]{K}$ if and only if $K \leq L$.
  \item \label{i:narrow}
    If $K/L$~is a chief factor of~$G$, there is a closed normal
    subgroup~$A$ which is narrow in~$G$ and is contained in~$K$ but
    not in~$L$.  This narrow subgroup satisfies $A \cap L = \Mel[G]{A}$.
  \end{enumerate}
\end{lemma}

It is well-known that a finitely generated finite-by-abelian discrete
group is centre-by-finite.  This is established by ideas related to
FC\nbd groups (see~\cite[Section~14.5]{Robinson}, in particular, the
proof of~(14.5.11) in that source).  In the case of profinite groups,
however, the hypothesis of finite generation is unnecessary, as
observed by Detomi, Morigi and Shumyatsky (see \cite[Lemma~2.7]{DMS}).
In fact, a similar argument establishes the following result needed in
our context and that only needs residual finiteness as a hypothesis:

\begin{lemma}
  Let $G$~be a residually finite group with a finite normal
  subgroup~$N$.  If $G/N$~is virtually class\nbd$c$ nilpotent, for
  some~$c \geq 0$, then $G$ is also virtually class\nbd$c$ nilpotent.
\end{lemma}

\begin{proof}
  For each non-identity element~$x$ of~$N$, there exists a normal
  subgroup of finite index in~$G$ that does not contain~$x$.  By
  intersecting these, we produce a normal subgroup~$K$ of finite index
  in~$G$ such that $N \cap K = \1$.  Then $G$~embeds in the direct
  product $G/N \times G/K$ of~$G/N$ and a finite group, so the result
  follows.
\end{proof}

\begin{cor}
  \label{cor:nofinite}
  Let $G$~be a profinite group that is \JNNcF\@.  Then $G$~has no
  non-trivial finite normal subgroup.
\end{cor}

Sections~\ref{sec:JNNF} and~\ref{sec:hJNNF} are concerned with
profinite \JNNcF\ groups, whereas the last two sections consider both
profinite and abstract \JNNcF\ groups.  To state efficiently the
results in Section~\ref{sec:finiteindex}, we shall adopt there the
convention that ``subgroup'' for a profinite group means ``closed
subgroup'' so that it remains in the same category.  For the results
in the current section that will be used in the discrete case, we
simply bracket the word ``closed'' to indicate it is unnecessary in
such a situation.  The following lemma illustrates this convention.
It is a standard elementary fact about just non\nbd$\P$ groups when
$\P$~is a property that is inherited by both finite direct products
and subgroups.

\begin{lemma}
  \label{lem:JNNF-intersect}
  Let $G$~be a profinite group or discrete group that is \JNNcF\@.  If
  $K$~and~$L$ are non-trivial (closed) normal subgroups of~$G$, then
  $K \cap L \neq \1$.
\end{lemma}

As Reid~\cite{Reid10a}, we shall use Wilson's concept~\cite{JSW00} of
a basal subgroup:

\begin{defn}
  A subgroup~$B$ of a group~$G$ is called \emph{basal} if $B$~is
  non-trivial, has finitely many conjugates $B_{1}$,~$B_{2}$,
  \dots,~$B_{n}$ in~$G$ and the normal closure of~$B$ in~$G$ is the
  direct product of these conjugates: $B^{G} = B_{1} \times B_{2}
  \times \dots \times B_{n}$.
\end{defn}

Lemma~\ref{lem:Reid-basal} below is based on~\cite[Lemma~5]{Reid10a}.
The hypothesis that $K$~has only finitely many conjugates is
sufficient to adapt the proof of Reid's lemma to our needs.  In its
statement, and in many that follow, we shall say that a (profinite or
discrete) group~$G$ is \emph{Fitting-free} when it has no non-trivial
abelian (closed) normal subgroup.  This is immediately equivalent to
the requirement that the Fitting subgroup~$\Fitt{G}$ be trivial.
Furthermore, if $G$~is a \JNNcF\ group, one observes that $G$~is
Fitting-free if and only if $G$~is not virtually soluble.  If $K$~is a
normal subgroup of~$G$, then $\Centre{K} = K \cap \Cent{G}{K}$ and we
deduce the following characterization of the Fitting-free condition in
\JNNcF\ groups using Lemma~\ref{lem:JNNF-intersect}.

\begin{lemma}
  \label{lem:FittingFree-Cent}
  Let $G$~be a profinite or discrete group that is \JNNcF\@.  Then
  $G$~is Fitting-free if and only if\/ $\Cent{G}{K} = \1$ for every
  non-trivial normal (closed) subgroup~$K$ of~$G$.
\end{lemma}

\begin{lemma}
  \label{lem:Reid-basal}
  Let $G$~be a profinite or discrete group that is Fitting-free.  Let
  $K$~be a non-trivial (closed) subgroup of~$G$ whose conjugates
  $\set{K_{i}}{i \in I}$ are parametrized by the finite set~$I$ and
  which satisfies $K \normal K^{G}$.  Then there exists some $J
  \subseteq I$ such that $\bigcap_{j \in J} K_{j}$ is basal.
\end{lemma}

\begin{proof}
  For $J \subseteq I$, define $K_{J} = \bigcap_{j \in J} K_{j}$.  Let
  $\mathcal{I}$~be the set of subsets~$J$ of~$I$ such that
  $K_{J} \neq \1$.  Certainly $\mathcal{I}$~is non-empty since it
  contains all singletons as $K \neq \1$.  Choose $J \in \mathcal{I}$
  of largest size and define $B = K_{J}$.  Then $B$~also has finitely
  many conjugates in~$G$ and we denote these by $B_{1}$,~$B_{2}$,
  \dots,~$B_{n}$.  Two distinct conjugates intersect trivially, $B_{i}
  \cap B_{j} = \1$ when $i \neq j$, since this is the intersection of
  more than~$\order{J}$ conjugates of~$K$.  Since each~$K_{i}$ is
  normal in~$K^{G}$, it follows that each~$B_{j} \normal K^{G}$ and
  therefore $[B_{i},B_{j}] \leq B_{i} \cap B_{j} = \1$ when $i \neq
  j$.  Set $L = B^{G} = B_{1}B_{2}\dots B_{n}$.  Then the centre
  of~$L$ is the product of the centres of the~$B_{j}$.  Our hypothesis
  that $G$~has no non-trivial abelian (closed) subgroup then forces
  $\Centre{B_{j}} = \1$ for each~$j$.  Now if $j \in \{1,2,\dots,n\}$,
  set $P_{j} = B_{1} \dots B_{j-1} B_{j+1} \dots B_{n}$.  Then
  $[B_{j},P_{j}] = \1$ and so $P_{j} \cap B_{j} \leq \Centre{B_{j}} =
  \1$.  Since this holds for each~$j$, we conclude that $L = B_{1}
  \times B_{2} \times \dots \times B_{n}$; that is, $B$~is basal.
\end{proof}


\subsection{Properties of virtually nilpotent profinite groups}

If $N$~is a closed normal subgroup of a profinite group~$G$, we define
the commutator subgroup $[N,_{i}G] \normalc G$ recursively by
$[N,_{0}G] = N$ and $[N,_{i}G] = [ [N,_{i-1}G], G]$ for $i \geq 1$.
Thus, using left-normed commutator notation,
\[
[N,_{c}G] = [N,\underbrace{G,G,\dots,G}_{\text{$c$~times}}].
\]
We also write~$\Centre[i]{G}$ for the $i$th term of the upper central
series of a group~$G$.

\begin{lemma}
  \label{lem:comms}
  Let $G$~be a finitely generated profinite group and $N$~be an open
  normal subgroup of~$G$ such that $\lowerc{c+1}{N} = \1$ for some $c
  \geq 0$.  Then $[N,_{i}G]$~is an open subgroup of~$\lowerc{c+1}{G}$
  for all $i \geq c$.
\end{lemma}

\begin{proof}
  Define $k = \order{G/N}$.  It follows from the definitions that
  $[N,_{c}G]$~is a closed normal subgroup of~$G$ with $N/[N,_{c}G]
  \leq \Centre[c]{G/[N,_{c}G]}$.  Hence, this term of the upper
  central series is open in~$G/[N,_{c}G]$ and a theorem of Baer (see
  \cite[(14.5.1)]{Robinson}) shows that $\lowerc{c+1}{G/[N,_{c}G]}$ is
  finite.  Hence $[N,_{c}G]$~is an open subgroup of~$\lowerc{c+1}{G}$.

  Now suppose that we have shown $[N,_{i}G]$~is open
  in~$\lowerc{c+1}{G}$ for some~$i \geq c$.  This subgroup is 
  generated, modulo~$[N,_{i+1}G]$, by all left-normed commutators
  $[x,y_{1},y_{2},\dots,y_{i}]$ where $x$~is selected from
  some finite generating set for~$N$ and $y_{1},y_{2},\dots,y_{i}$
  from a finite generating set for~$G$.  In particular,
  $[N,_{i}G]/[N,_{i+1}G]$~is a finitely generated abelian profinite
  group.  Furthermore, standard commutator calculus shows that,
  modulo~$[N,_{i+1}G]$,
  \[
    [x,y_{1},y_{2},\dots,y_{i}]^{k^{i}} \equiv
    [x,y_{1}^{\, k},y_{2}^{\, k},\dots,y_{i}^{\, k}] \in
    \lowerc{i+1}{N} = \1.
  \]
  Hence, every generator of~$[N,_{i}G]/[N,_{i+1}G]$ has finite order
  and so this abelian group is finite.  It follows that
  $[N,_{i+1}G]$~is an open subgroup of~$[N,_{i}G]$.  The lemma then
  follows by induction on~$i \geq c$.
\end{proof}

The next result establishes, in particular,
Theorem~\ref{thm:maintool}, stated in the introduction.

\begin{thm}
  \label{thm:Cset-finite}
  \begin{enumerate}
  \item \label{i:Cset-finite}
    Let $G$~be a finitely generated virtually class\nbd$c$ nilpotent
    profinite group, for some~$c \geq 0$.  Then the set $\set{
      \lowerc{c+1}{K} }{K \normalc G}$ is finite.

  \item \label{i:Aset-converse}
    Conversely, if $G$~is a profinite group such that $\set{
      \lowerc{c+1}{K} }{K \normalo G}$ is finite, for some~$c \geq 0$,
    then $G$ is virtually class\nbd$c$ nilpotent.
  \end{enumerate}
\end{thm}

\begin{proof}
  \ref{i:Cset-finite}~Let $N$~be an open normal subgroup of~$G$ such
  that $\lowerc{c+1}{N} = \1$.  Let $K$~be any closed normal subgroup
  of~$G$ and set $L = KN$.  By standard commutator calculus, any
  element of~$[N,_{2c}KN]$ can be expressed as a product of
  commutators $[x,y_{1},y_{2},\dots,y_{2c}]$ where $x \in N$ and
  each~$y_{i}$ belongs either to~$K$ or to~$N$.  Since such a
  commutator involves either at least $c+1$~entries from~$K$ or at
  least~$c+1$~entries from~$N$, we deduce
  \[
  [N,_{2c}L] = [N,_{2c}KN] \leq \lowerc{c+1}{N} \, [N,_{c+1}K] \leq
  \lowerc{c+1}{K} \leq \lowerc{c+1}{L}.
  \]
  Furthermore, upon applying Lemma~\ref{lem:comms} to the profinite
  group~$L$, we conclude that $[N,_{2c}L]$~is an open subgroup
  of~$\lowerc{c+1}{L}$.  Therefore, for each open normal subgroup~$L$
  of~$G$ that contains~$N$, there are at most finitely many
  possibilities for~$\lowerc{c+1}{K}$ as $K$~ranges over all closed
  normal subgroups of~$G$ with $KN = L$.  Finally, since there are
  only finitely many possibilities for~$L$, we conclude that $\set{
    \lowerc{c+1}{K} }{K \normalc G}$ is indeed finite.

  \ref{i:Aset-converse}~Let $G$~be a profinite group and suppose that
  $\mathcal{A} = \set{ \lowerc{c+1}{K} }{ K \normalo G }$ is finite.
  If $N$~is any open normal subgroup of~$G$, then the set
  $\mathcal{L}_{G/N} = \set{ \lowerc{c+1}{H} }{ H \normal G/N }$ is
  the image of~$\mathcal{A}$ under the map induced by the natural
  homomorphism $G \to G/N$.  In particular, there exists some open
  normal subgroup~$M$ of~$G$ such that $\order{\mathcal{L}_{G/M}}$~is
  maximal.  If $N$~is an open normal subgroup of~$G$ contained in~$M$,
  then $\order{\mathcal{L}_{G/N}} = \order{\mathcal{L}_{G/M}}$ and so,
  in particular, $\lowerc{c+1}{M/N}$~must coincide
  with~$\lowerc{c+1}{N/N}$; that is, $\lowerc{c+1}{M} \leq N$.  As
  this holds for all such open normal subgroups~$N$, we conclude that
  $\lowerc{c+1}{M} = \1$.  This shows that $G$~is virtually
  class\nbd$c$ nilpotent.
\end{proof}

The following example demonstrates that the assumption of finite
generation is necessary in
Theorem~\ref{thm:Cset-finite}\ref{i:Cset-finite}.  We construct a
countably-based virtually abelian pro\nbd$p$ group such that the set
$\set{ K' }{K \normalo G}$ contains infinitely many subgroups.

\begin{example}
  Let $p$~be a prime and, for each~$i \geq 0$, set~$V_{i}$ to be the
  direct product of $p$~copies of the cyclic group~$C_{p}$ of
  order~$p$.  Take $H = C_{p}$ and let $H$~act on each~$V_{i}$ by
  cyclically permuting the factors.  Define $W_{i} = V_{i} \rtimes H
  \cong C_{p} \wr C_{p}$, the standard wreath product.  Then
  $[V_{i},H]$~and~$[V_{i},H,H]$ are normal subgroups of~$W_{i}$ of
  indices $p^{2}$~and~$p^{3}$, respectively.  Now take $G = \left(
  \prod_{i=0}^{\infty} V_{i} \right) \rtimes H$.  This is a virtually
  abelian pro\nbd$p$ group, indeed $G = \invlim G_{n}$ where $G_{n} =
  \left( \prod_{i=0}^{n} V_{i} \right) \rtimes H$.  Certainly $G$~is
  not finitely generated.  Observe that, for each finite subset~$S$
  of~$\Nat_{0}$,
  \[
  U_{S} = \biggl( \prod_{i \in S} [V_{i},H] \times \prod_{i \notin S}
  V_{i} \biggr) \rtimes H
  \]
  is an open normal subgroup of~$G$ and
  \[
  U_{S}' = \prod_{i \in S} [V_{i},H,H] \times \prod_{i \notin S} [V_{i},H].
  \]
  In particular, the set $\set{K'}{K \normalo G}$ is infinite for this
  group~$G$.
\end{example}


\section{Characterization of \JNNcF\ profinite groups}
\label{sec:JNNF}

We fix the integer $c \geq 0$ throughout this section.  In order to
establish Theorem~\ref{thm:obliq} that characterizes Fitting-free
\JNNcF\ profinite groups, we shall associate a directed graph~$\Gamma$
to the set~$\Cset{H}$ that appears in the statement of
Theorem~\ref{thm:Aset} below.  This graph is similar to that used by
Reid in~\cite{Reid10b}.  A key difference is that the vertices
of~$\Gamma$ correspond only to closed subgroups that have the
form~$\lowerc{c+1}{K}$ (where $K$~is a closed normal subgroup of the
profinite group under consideration) rather than any other non-trivial
closed subgroups that the group may have.  We begin by describing this
graph and establishing that it is locally finite.

In the following, recall that the Melnikov subgroup~$\Mel{N}$ of~$N$
is the intersection of the maximal open normal subgroups of~$N$.

\begin{lemma}
  \label{lem:JNNF-Mel}
  Let $G$~be a Fitting-free \JNNcF\ profinite group and let $N$~be a
  non-trivial closed normal subgroup of~$G$.  Then
  $\lowerc{i}{\Mel{N}} \neq \1$ for all~$i \geq 1$.
\end{lemma}

\begin{proof}
  We shall show that the normal subgroup~$\Mel{N}$ is non-trivial, for
  the hypothesis that $G$~is Fitting-free then ensures it cannot be
  nilpotent.  Suppose for a contradiction that $\Mel{N} = \1$.  Let
  $\mathcal{L}$~be the set of open normal subgroups~$M$ of~$N$ such
  that $N/M$~is cyclic of prime order and $\mathcal{M}$~be the set of
  open normal subgroups~$M$ of~$N$ such that $N/M$~is a non-abelian
  finite simple group.  Then $\bigl( \bigcap \mathcal{L} \bigr) \cap
  \bigl( \bigcap \mathcal{M} \bigr) = \Mel{N} = \1$.  By
  Lemma~\ref{lem:JNNF-intersect}, either $\bigcap \mathcal{L} = \1$ or
  $\bigcap \mathcal{M} = \1$.  If $\bigcap \mathcal{L} = \1$, then
  $N$~embeds in a Cartesian product of cyclic groups of prime order
  and so $N$~would be abelian, contrary to hypothesis.

  Hence $\bigcap \mathcal{M} = \1$.  Then \cite[Corollary~8.2.3]{RZ}
  tells us that $N$~is a Cartesian product of non-abelian finite
  simple groups~$S_{R}$ indexed by the set~$\mathcal{M}$, say $N =
  \prod_{R \in \mathcal{M}} S_{R}$.  Now there exists some open normal
  subgroup~$K$ of~$G$ such that $N \cap K < N$.  Define
  \[
  M_{1} = \prod_{S_{R} \leq K} S_{R}
  \AND
  M_{2} = \prod_{S_{R} \nleq K} S_{R},
  \]
  the products of those factors~$S_{R}$ contained in~$K$ and not
  contained in~$K$, respectively.  Any closed normal subgroup of~$N$
  is the product of the factors~$S_{R}$ that it contains, so
  $M_{1} = N \cap K$.  Hence $M_{2}$~is non-trivial and finite.
  Furthermore, since $K$~is normal in~$G$, for $g \in G$, \
  $S_{R} \nleq K$ if and only if $S_{R}^{\,\, g} \nleq K$.  Therefore
  $M_{2}$~is a normal subgroup of~$G$ and we have a contradiction by
  Corollary~\ref{cor:nofinite}.
\end{proof}

Let $G$~be a finitely generated Fitting-free \JNNcF\ profinite group
and $H$~be an open subgroup of~$G$.  Construct a directed graph
$\Gamma = \Gamma(H)$ whose vertices are the members of the set
\[
\Cset{H} = \set{ \lowerc{c+1}{K} }{\text{$K \normalc G$ with
    $\lowerc{c+1}{K} \nleq H$}}
\]
and where there is an edge from a member~$A$ of~$\Cset{H}$ to another
member~$B$ when $B < A$ and there is no $C \in \Cset{H}$ with $B < C <
A$.

\begin{lemma}
  \label{lem:JNNF-graph}
  Let $G$~be a finitely generated Fitting-free \JNNcF\ profinite
  group, let $H$~be an open subgroup of~$G$ and let $\Gamma =
  \Gamma(H)$ be the graph defined above.
  \begin{enumerate}
  \item \label{i:graph-Mel}
    If $K$~and~$L$ are closed normal subgroups of~$G$ such that
    $\lowerc{c+1}{K}, \lowerc{c+1}{L} \in \Cset{H}$ and there is an
    edge from~$\lowerc{c+1}{K}$ to~$\lowerc{c+1}{L}$ in~$\Gamma$,
    then $\lowerc{c+1}{\Mel{\lowerc{c+1}{K}}} \leq \lowerc{c+1}{L}$.
  \item \label{i:graph-edges}
    If $K$~is a closed normal subgroup of~$G$ such that
    $\lowerc{c+1}{K} \in \Cset{H}$, then there are at most finitely
    many $\lowerc{c+1}{L} \in \Cset{H}$ such that there is an edge
    from~$\lowerc{c+1}{K}$ to~$\lowerc{c+1}{L}$ in~$\Gamma$.
  \end{enumerate}
\end{lemma}

\begin{proof}
  \ref{i:graph-Mel}~Suppose that there is an edge
  from~$\lowerc{c+1}{K}$ to~$\lowerc{c+1}{L}$ in~$\Gamma$.  Then
  $\lowerc{c+1}{L}$~is a proper subgroup of~$\lowerc{c+1}{K}$, so the
  intersection~$R$ of the maximal open normal subgroups
  of~$\lowerc{c+1}{K}$ that contain~$\lowerc{c+1}{L}$ satisfies $R <
  \lowerc{c+1}{K}$.  By definition, $\Mel{\lowerc{c+1}{K}} \leq R$
  and so $\Mel{\lowerc{c+1}{K}} \lowerc{c+1}{L} \leq R <
  \lowerc{c+1}{K}$.  Take $J = \Mel{\lowerc{c+1}{K}}L$.  Then
  $J$~is a closed normal subgroup of~$G$ and $\lowerc{c+1}{L} \leq
  \lowerc{c+1}{J} \leq \Mel{\lowerc{c+1}{K}} \lowerc{c+1}{L} <
  \lowerc{c+1}{K}$.  Since there is an edge in~$\Gamma$
  from~$\lowerc{c+1}{K}$ to~$\lowerc{c+1}{L}$, this forces
  $\lowerc{c+1}{J} = \lowerc{c+1}{L}$ and hence
  $\lowerc{c+1}{\Mel{\lowerc{c+1}{K}}} \leq \lowerc{c+1}{L}$.

  \ref{i:graph-edges}~Define
  $M = \lowerc{c+1}{\Mel{\lowerc{c+1}{K}}}$.  By
  Lemma~\ref{lem:JNNF-Mel}, $M \neq \1$ and hence $Q = G/M$ is
  virtually class\nbd$c$ nilpotent.  If there is an edge
  from~$\lowerc{c+1}{K}$ to~$\lowerc{c+1}{L}$ in~$\Gamma$ then, by
  part~\ref{i:graph-Mel}, $\lowerc{c+1}{L}$~corresponds to
  $\lowerc{c+1}{L/M}$ and here $L/M$~is a closed normal subgroup
  of~$Q$.  Consequently, there are only finitely many possibilities
  for~$\lowerc{c+1}{L}$ by
  Theorem~\ref{thm:Cset-finite}\ref{i:Cset-finite}.
\end{proof}

\begin{thm}
  \label{thm:Aset}
  Let $G$~be a finitely generated infinite profinite group that is
  Fitting-free and let $c$~be a non-negative integer.  Then the
  following conditions are equivalent:
  \begin{enumerate}
  \item $G$~is \JNNcF;
  \item \label{i:Aset-condn}
    the set $\Aset{H} = \set{ \lowerc{c+1}{K} }{\text{$K \normalo G$
        with $\lowerc{c+1}{K} \nleq H$}}$ is finite for every open
    subgroup~$H$ of~$G$;
  \item \label{i:Cset-condn}
    the set $\Cset{H} = \set{ \lowerc{c+1}{K} }{\text{$K \normalc G$
        with $\lowerc{c+1}{K} \nleq H$}}$ is finite for every open
    subgroup~$H$ of~$G$.
  \end{enumerate}
\end{thm}

Observe that if $H$~is any open subgroup of~$G$ with $C =
\Core{G}{H}$, then $\Aset{H} = \Aset{C}$ and $\Cset{H} = \Cset{C}$.
Hence each of the Conditions~\ref{i:Aset-condn}
and~\ref{i:Cset-condn} is equivalent to the requirement that the given
set be finite for every open \emph{normal} subgroup~$H$ of~$G$.

\begin{proof}
  Assume that $G$~is \JNNcF\@.  Suppose that $\Cset{H}$~is infinite
  for some open subgroup~$H$ of~$G$.  As described above, construct
  the graph $\Gamma = \Gamma(H)$ whose vertices are the members
  of~$\Cset{H}$.  Lemma~\ref{lem:JNNF-graph}\ref{i:graph-edges}~tells
  us that each vertex of~$\Gamma$ has finite out-degree.  Furthermore,
  if $\lowerc{c+1}{K} \in \Cset{H}$, then $G/\lowerc{c+1}{K}$~is a
  proper quotient of~$G$ and so is virtually class\nbd$c$ nilpotent.
  Hence, by Theorem~\ref{thm:Cset-finite}\ref{i:Cset-finite},
  $G/\lowerc{c+1}{K}$~contains only finitely many subgroups of the
  form~$\lowerc{c+1}{\bar{L}}$ where $\bar{L}$~is a closed normal
  subgroup; that is, there are only finitely many members
  of~$\Cset{H}$ that contain~$\lowerc{c+1}{K}$.  Consequently there is
  a path of finite length in~$\Gamma$ from~$\lowerc{c+1}{G}$
  to~$\lowerc{c+1}{K}$.

  Thus $\Gamma$ is a connected, locally finite, infinite directed
  graph.  By K\"{o}nig's Lemma (compare, for example,
  \cite[Lemma~8.1.2]{Diestel}), $\Gamma$~has an infinite directed path
  and this corresponds to an infinite descending chain
  $\lowerc{c+1}{K_{1}} > \lowerc{c+1}{K_{2}} > \dots$ of members
  of~$\Cset{H}$.  An application of~\cite[Lemma~2.4]{Reid10b}, taking
  $O = H$, shows that $J = \bigcap_{i=1}^{\infty} \lowerc{c+1}{K_{i}}
  \neq \1$.  Then $G/J$~is finitely generated and virtually
  class\nbd$c$ nilpotent but it has infinitely many subgroups of the
  form~$\lowerc{c+1}{K_{i}}/J$ with $K_{i} \normalc G$.  This
  contradicts Theorem~\ref{thm:Cset-finite}\ref{i:Cset-finite}.  We
  conclude therefore that $\Cset{H}$~is finite for every open
  subgroup~$H$ of~$G$.

  Since $\Aset{H} \subseteq \Cset{H}$ for every~$H$, it is certainly
  the case that the third condition in the statement implies the
  second.

  Suppose finally that $\Aset{H}$~is finite for every open
  subgroup~$H$ of~$G$.  As $G$~is Fitting-free, it is not virtually
  nilpotent.  Let $N$~be a non-trivial closed normal subgroup of~$G$.
  Then $\lowerc{c+1}{N} \neq \1$ and so there exists an open normal
  subgroup~$H$ of~$G$ such that $\lowerc{c+1}{N} \nleq H$.  By
  hypothesis, $\Aset{H} = \{ \lowerc{c+1}{L_{1}}, \lowerc{c+1}{L_{2}},
  \dots, \lowerc{c+1}{L_{r}} \}$ for some open normal subgroups
  $L_{1}$,~$L_{2}$, \dots,~$L_{r}$ of~$G$.  Set $L = \bigcap_{i=1}^{r}
  L_{i}$.  If $K$~is an open normal subgroup of~$G$ with $N \leq K$,
  then necessarily $\lowerc{c+1}{K} \nleq H$ and so $\lowerc{c+1}{K} =
  \lowerc{c+1}{L_{i}}$ for some~$i$.  Therefore
  \[
  \lowerc{c+1}{L} \leq
  \bigcap \set{ \lowerc{c+1}{K} }{N \leq K \normalo G} \leq
  \bigcap \set{ K }{N \leq K \normalo G} = N.
  \]
  Hence $LN/N$~is a class\nbd$c$ nilpotent open subgroup of~$G/N$, as
  required.
\end{proof}

We shall now use Theorem~\ref{thm:Aset} to establish further
information about finitely generated Fitting-free \JNNcF\ groups,
including a description of them as inverse limits of suitable
virtually nilpotent groups (see Theorem~\ref{thm:JNNF-invlim} below).

Suppose that $G$~is a finitely generated Fitting-free \JNNcF\ group.
We start with any open normal subgroup~$H_{0}$.  Then certainly
$\lowerc{c+1}{H_{0}} \neq \1$ since $G$~is Fitting-free.  Now assume,
as an inductive hypothesis, that we have constructed a sequence of
open normal subgroups $G \geq H_{0} > H_{1} > \dots > H_{n-1}$ such
that for each $i \in \{1, 2, \dots, n-1\}$ the following holds:
$\lowerc{c+1}{H_{i}} \leq \Mel[G]{\lowerc{c+1}{H_{i-1}}}$ and if
$N$~is an open normal subgroup of~$G$ either $\lowerc{c+1}{N} \leq
H_{i-1}$ or $\lowerc{c+1}{H_{i}} \leq \lowerc{c+1}{N}$.  By
Theorem~\ref{thm:Aset}, the set
$\Aset{H_{n-1}} = \set{ \lowerc{c+1}{K} }{\text{$K \normalo G$ with
    $\lowerc{c+1}{K} \nleq H_{n-1}$}}$ is finite.  Also
$\Mel[G]{\lowerc{c+1}{H_{n-1}}} \neq \1$ since it
contains~$\Mel{\lowerc{c+1}{H_{n-1}}}$ which is non-trivial by
Lemma~\ref{lem:JNNF-Mel}.  Let
\[
R = \Mel[G]{\lowerc{c+1}{H_{n-1}}} \cap \bigcap \Aset{H_{n-1}}.
\]
Since this is a finite intersection of non-trivial closed normal
subgroups, $R$~is also a non-trivial closed normal subgroup of~$G$ by
Lemma~\ref{lem:JNNF-intersect}.  Then $G/R$~is virtually class\nbd$c$
nilpotent, so there exists an open normal subgroup~$S$ of~$G$ with
$\lowerc{c+1}{S} \leq R$.  Take $H_{n} = H_{n-1} \cap S$, so that
$H_{n}$~is open in~$G$, \ $H_{n} \leq H_{n-1}$ and
\[
\lowerc{c+1}{H_{n}} \leq R \leq
\Mel[G]{\lowerc{c+1}{H_{n-1}}} < \lowerc{c+1}{H_{n-1}}.
\]
Furthermore, if $N$~is an open normal subgroup of~$G$,
then either $\lowerc{c+1}{N} \leq H_{n-1}$ or $\lowerc{c+1}{N} \in
\Aset{H_{n-1}}$.  In the latter case, $\lowerc{c+1}{H_{n}} \leq R \leq
\lowerc{c+1}{N}$ according to our definition of~$R$.

By repeated application of these steps, we obtain a descending
sequence of open normal subgroups~$H_{n}$.  Let
$J = \bigcap_{n=0}^{\infty} H_{n}$.  If $J \neq \1$, then necessarily
$\lowerc{c+1}{J} \neq \1$ so $G/\lowerc{c+1}{J}$~is virtually
class\nbd$c$ nilpotent.  By
Theorem~\ref{thm:Cset-finite}\ref{i:Cset-finite}, the set
$\set{ \lowerc{c+1}{K} }{ K \normalo G/\lowerc{c+1}{J} }$ is finite
but each term $\lowerc{c+1}{H_{i}}/\lowerc{c+1}{J}$ is a member of
this set.  This contradiction shows that $J = \1$.

In conclusion, we have established the following observation:

\begin{lemma}
  \label{lem:fullstruct}
  Let $G$~be a finitely generated profinite group that is Fitting-free
  and \JNNcF\@. Then there is a descending sequence $G \geq H_{0} >
  H_{1} > H_{2} > \dots$ of open normal subgroups such that
  \begin{enumerate}
  \item for each $n \geq 1$, \ $\lowerc{c+1}{H_{n}} \leq
    \Mel[G]{\lowerc{c+1}{H_{n-1}}} < \lowerc{c+1}{H_{n-1}}$,
  \item $\bigcap_{n=0}^{\infty} H_{n} = \1$;
  \item if $N$~is an open normal subgroup of~$G$ and $n \geq 1$, then
    either $\lowerc{c+1}{N} \leq H_{n-1}$ or $\lowerc{c+1}{H_{n}} \leq
    \lowerc{c+1}{N}$. 
  \end{enumerate}
\end{lemma}

The conditions appearing in the lemma are sufficient to ensure that
the group~$G$ is \JNNcF\@.  In fact, we can make them marginally
weaker as the following shows:

\begin{thm}
  \label{thm:JNNFstruct}
  Let $G$~be a finitely generated Fitting-free profinite group and let
  $c$~be a non-negative integer.  Then $G$~is \JNNcF\ if and only if
  there is a descending sequence $G \geq H_{0} > H_{1} > H_{2} >
  \dots$ of open normal subgroups such that
  \begin{enumerate}
  \item \label{i:JNNFstruct-intersect}
    $\bigcap_{n=0}^{\infty} H_{n} = \1$, and
  \item \label{i:JNNF-dichotomy}
    if $N$~is an open normal subgroup of~$G$ and $n \geq 1$, then
    either $\lowerc{c+1}{N} \leq H_{n-1}$ or $\lowerc{c+1}{H_{n}} \leq
    \lowerc{c+1}{N}$.
  \end{enumerate}
  Furthermore, when these conditions are satisfied, the group~$G$ is
  just infinite if and only if $\lowerc{c+1}{H_{n}}$~has finite index
  in~$G$ for all $n \geq 0$.
\end{thm}

\begin{proof}
  If $G$~is \JNNcF, the existence of the descending sequence of open
  subgroups~$H_{n}$ is provided by Lemma~\ref{lem:fullstruct}.
  Suppose conversely that $G$~possesses a descending chain~$H_{n}$, $n
  \geq 0$, of open normal subgroups satisfying
  \ref{i:JNNFstruct-intersect} and~\ref{i:JNNF-dichotomy}.  Since
  $G$~is Fitting-free, it cannot be virtually nilpotent.  Let $K$~be a
  non-trivial closed normal subgroup of~$G$.  Then, for the same
  reason, $\lowerc{c+1}{K} \neq \1$.  Therefore, since
  condition~\ref{i:JNNFstruct-intersect} holds, there exists some~$m
  \geq 0$ such that $\lowerc{c+1}{K} \nleq H_{m}$.  Let $N$~be any
  open normal subgroup of~$G$ with $K \leq N$.  Since $\lowerc{c+1}{N}
  \nleq H_{m}$, Condition~\ref{i:JNNF-dichotomy} shows that
  $\lowerc{c+1}{H_{m+1}} \leq \lowerc{c+1}{N} \leq N$.  It follows
  that
  \[
  \lowerc{c+1}{H_{m+1}} \leq \bigcap \set{N}{ K \leq N \normalo G} = K
  \]
  and hence $G/K$~is virtually nilpotent of class~$c$, as required.

  Finally, observe that if $\order{G:\lowerc{c+1}{H_{n}}}$~is infinite
  for some $n \geq 0$, then $G/\lowerc{c+1}{H_{n}}$ is an infinite
  quotient and so $G$~is not just infinite.  On the other hand, if
  $\order{G:\lowerc{c+1}{H_{n}}} < \infty$ for all $n \geq 0$, then in
  the previous paragraph it follows that if $K$~is a non-trivial
  closed normal subgroup of~$G$ then $\lowerc{c+1}{H_{m+1}} \leq K$
  for some $m \geq 0$ and so $K$~has finite index.  Hence $G$~is in
  fact just infinite under this assumption.
\end{proof}

In~\cite[Theorem~3.6]{Reid12}, Reid presents a condition which
guarantees the existence of a just infinite quotient of a profinite
group.  The condition is expressed in terms of the relation~$\succnar$
concerning chief factors of the profinite group~$G$ under
consideration.  Notice, however, that with use
of~\cite[Proposition~3.5(iii)]{Reid12}, the assumption that
$K_{1}/L_{1} \succnar K_{2}/L_{2} \succnar \dots$ is a descending
sequence of open chief factors (as appears
in~\cite[Theorem~3.6]{Reid12}) is equivalent to the existence of open
normal subgroups $G \geq K_{1} > L_{1} \geq K_{2} > L_{2} \geq \dots$
with $L = \bigcap_{n=1}^{\infty} L_{n}$ such that, for each~$n$, \
$K_{n}/L$~is a narrow subgroup of~$G/L$ and
$\Mel[G/L]{K_{n}/L} = L_{n}/L$.  Theorem~\ref{thm:JNNFquot} below can
consequently be viewed as an analogous result for the existence of
\JNNcF\ quotients of a profinite group.

The application of Zorn's Lemma in our proof is more delicate than for
Reid's result.  Under the hypotheses and notation
of~\cite[Theorem~3.6]{Reid12}, the quotient~$G/L_{n}$ would be finite
and so would have only finitely many subgroups.  However, our
corresponding quotient~$G/L_{n}$ is a finitely generated virtually
nilpotent group and this does not necessarily even possess the
ascending chain condition on closed normal subgroups.

\begin{thm}
  \label{thm:JNNFquot}
  Let $G$~be a finitely generated profinite group and let $c$~be a
  non-negative integer.
  \begin{enumerate}
  \item \label{i:JNNFquot-exist}
    For each~$n \geq 1$, let $K_{n}$~and~$L_{n}$ be closed normal
    subgroups of~$G$ and define $L = \bigcap_{n=1}^{\infty} L_{n}$.
    Suppose that
    \[
    G \geq K_{1} > L_{1} \geq \lowerc{c+1}{L_{1}}L \geq K_{2} > L_{2}
    \geq \lowerc{c+1}{L_{2}}L \geq \dots
    \]
    and that, for each~$n$, \ $K_{n}/L$~is a narrow subgroup of~$G/L$
    with $\Mel[G/L]{K_{n}/L} = L_{n}/L$ and $G/L_{n}$~is virtually
    class\nbd$c$ nilpotent.  Then there exists a closed normal
    subgroup~$K$ of~$G$ that is maximal subject to the conditions that
    $K \geq L$ and $K_{n} \nleq L_{n}K$ for all~$n$.  Furthermore,
    such a closed normal subgroup~$K$ has the property that $G/K$~is
    \JNNcF.

  \item \label{i:JNNFquot-conv}
    Every Fitting-free \JNNcF\ quotient~$G/K$ of~$G$ arises in the
    manner described in~\ref{i:JNNFquot-exist} with $L = K$.
  \end{enumerate}
\end{thm}

\begin{proof}
  \ref{i:JNNFquot-exist}~When $c = 0$, this follows
  from~\cite[Theorem~3.6(i)]{Reid12}.  We shall assume that $c \geq 1$
  in the following argument.  Let $\mathcal{N}$~be the set of all
  closed normal subgroups~$N$ of~$G$ which contain~$L$ and such that
  $K_{n} \nleq L_{n}N$ for all~$n \geq 1$.  We shall
  order~$\mathcal{N}$ by inclusion.  Observe that $L \in \mathcal{N}$
  since $L_{n}L = L_{n} < K_{n}$.  Let $\mathcal{C}$~be a chain
  in~$\mathcal{N}$ and define $R = \overline{ \bigcup \mathcal{C} }$.
  Suppose that $R \notin \mathcal{N}$.  Then there exists some~$m \geq
  1$ such that $K_{m} \leq L_{m}R$.  If $C \in \mathcal{C}$, then
  $K_{m} \nleq L_{m}C$, so $(K_{m} \cap C)L_{m} = K_{m} \cap L_{m}C <
  K_{m}$ and therefore $K_{m} \cap C \leq L_{m}$ since $L_{m}$~is
  maximal among $G$\nbd invariant open subgroups of~$K_{m}$.  Hence
  $[C,K_{m}] \leq L_{m}$ and so $C \leq \Cent{G}{K_{m}/L_{m}}$ for all
  $C \in \mathcal{C}$.  Since this centralizer is an open subgroup
  of~$G$, it follows that $R \leq \Cent{G}{K_{m}/L_{m}}$.  Hence
  $K_{m} \leq L_{m}R \leq \Cent{G}{K_{m}/L_{m}}$ and so the chief
  factor~$K_{m}/L_{m}$ is abelian; that is, it is an elementary
  abelian $q$\nbd group for some prime~$q$.

  Since $G/L_{m}$~is virtually class\nbd$c$ nilpotent, there is an
  open normal subgroup~$A$ with $L_{m} \leq A$ such that
  $\lowerc{c+1}{A} \leq L_{m}$.  If $K_{m} \nleq A$, then $K_{m} \cap
  A = L_{m}$ and so $K_{m}A/L_{m} \cong K_{m}/L_{m} \times A/L_{m}$ is
  also class\nbd$c$ nilpotent.  Consequently, if necessary, we may
  replace~$A$ by~$K_{m}A$ and hence assume $K_{m} \leq A$.  For each
  prime~$p$, write~$A[p]/L_{m}$ for the Sylow pro\nbd$p$ subgroup
  of~$A/L_{m}$.  Then $A/L_{m}$~is the product $\prod_{p} A[p]/L_{m}$
  of these pro\nbd$p$ groups.  Furthermore, for each $C \in
  \mathcal{C}$ and prime~$p$, let $C[p]/L_{m}$~be the Sylow pro\nbd$p$
  subgroup of $(C\cap A)L_{m}/L_{m}$.  Since $\mathcal{C}$~is a chain,
  so is the set $\mathcal{S}_{p} = \set{ C[p]/L_{m} }{C \in
    \mathcal{C}}$.  As a finitely generated nilpotent pro\nbd$p$
  group, $A[p]/L_{m}$~satisfies the ascending chain condition on
  closed subgroups and so there exists some maximal
  member~$M[p]/L_{m}$ of~$\mathcal{S}_{p}$.

  If it were the case that $K_{m}/L_{m} \leq M[q]/L_{m}$, then
  $K_{m} \leq (C \cap A)L_{m} \leq CL_{m}$ for some $C \in
  \mathcal{C}$, contrary to the fact that $C \in \mathcal{N}$.  Define
  $M$~to be the closed subgroup of~$A$ defined by $M/L_{m} = \prod_{p}
  M[p]/L_{m}$.  Then $K_{m} \nleq M$ since we have observed that
  $K_{m}/L_{m}$~is not contained in the Sylow pro\nbd$q$ subgroup
  of~$M/L_{m}$.

  On the other hand, $C \cap A \leq M$ for all $C \in \mathcal{C}$,
  since by construction $C[p]/L_{m} \leq M[p]/L_{m}$ for each
  prime~$p$.  Furthermore, since $A$~is a clopen subset of~$G$,
  \[
  \overline{ \bigcup_{C \in \mathcal{C}} (C \cap A) } = \overline{
    \left( \bigcup \mathcal{C} \right) \cap A } = R \cap A
  \]
  and so we conclude that $R \cap A \leq M$.  Consequently, $K_{m}
  \leq L_{m}R \cap A = (R \cap A)L_{m} \leq M$, which contradicts our
  previous observation.

  In conclusion, we have shown that $R = \overline{ \bigcup
    \mathcal{C} } \in \mathcal{N}$ and so every chain in~$\mathcal{N}$
  has an upper bound.  Therefore, by Zorn's Lemma, there is a maximal
  member~$K \in \mathcal{N}$; that is, $K$~is maximal subject to the
  condition that $K_{n} \nleq L_{n}K$ for all~$n \geq 1$.  Suppose
  that $G/K$~is virtually class\nbd$c$ nilpotent.  By
  Theorem~\ref{thm:Cset-finite}\ref{i:Cset-finite}, the set $\set{
    \lowerc{c+1}{J} }{J \normalc G/K}$ is finite.  Hence
  $\lowerc{c+1}{L_{m}}K = \lowerc{c+1}{L_{m+1}}K$ for some~$m \geq 1$
  and so
  \[
  K_{m+1} \leq \lowerc{c+1}{L_{m}}L \leq \lowerc{c+1}{L_{m}}K =
  \lowerc{c+1}{L_{m+1}}K \leq L_{m+1}K,
  \]
  contrary to the fact that $K \in \mathcal{N}$.  We deduce that
  $G/K$~is not virtually class\nbd$c$ nilpotent.

  Now let $N$~be a closed normal subgroup of~$G$ that strictly
  contains~$K$.  Then $N \notin \mathcal{N}$ by maximality of~$K$, so
  there exists some~ $m \geq 1$ such that $K_{m} \leq L_{m}N$; that
  is, $K_{m}/L \leq \Mel[G/L]{K_{m}/L} \cdot (N/L)$.
  Lemma~\ref{lem:Mel}\ref{i:Mel-incl} then tells us that $K_{m} \leq
  N$.  Hence $G/N$~is a quotient of~$G/L_{m}$ and so is virtually
  class\nbd$c$ nilpotent.  This shows that $G/K$~is indeed \JNNcF.

  \ref{i:JNNFquot-conv}~Assume that $G/K$~is a \JNNcF\ quotient of~$G$
  and that it is Fitting-free.  We define the sequences of closed
  normal subgroups $K_{n}$~and~$L_{n}$ as follows.  First take any
  chief factor of~$G/K$ and let $K_{1}/K$~be a narrow subgroup as
  provided by Lemma~\ref{lem:Mel}\ref{i:narrow} and define~$L_{1}$ by
  $L_{1}/K = \Mel[G/K]{K_{1}/K}$.  Note that $L_{1} > K$ by use of
  Corollary~\ref{cor:nofinite} and hence $\lowerc{c+1}{L_{1}}K > K$ by
  the hypothesis that $G/K$~is Fitting-free.  Assuming that, for
  some $n \geq 2$, we have defined $K_{n-1}$~and~$L_{n-1}$ with
  $\lowerc{c+1}{L_{n-1}}K > K$, use Lemma~\ref{lem:Mel} again to
  produce a narrow subgroup~$K_{n}/K$ of~$G/K$ with $K_{n} \leq
  \lowerc{c+1}{L_{n-1}}K$.  Define~$L_{n}$ by $L_{n}/K =
  \Mel[G/K]{K_{n}/K}$ and note $\lowerc{c+1}{L_{n}}K > K$.  In
  conclusion, this produces the required chain of closed normal
  subgroups
  \[
  G \geq K_{1} > L_{1} \geq \lowerc{c+1}{L_{1}}K \geq K_{2} > L_{2}
  \geq \lowerc{c+1}{L_{2}}K \geq \dots.
  \]
  Now $L = \bigcap_{n=1}^{\infty} L_{n}$ certainly contains~$K$, while
  the quotient~$G/L$ cannot be virtually class\nbd$c$ nilpotent by
  Theorem~\ref{thm:Cset-finite}\ref{i:Cset-finite} as the
  subgroups~$\lowerc{c+1}{L_{n}/L}$ are distinct.  Hence $L = K$.
  Finally, if $N$~is a closed normal subgroup of~$G$ with $N > K$,
  then $G/N$~is virtually class\nbd$c$ nilpotent and so, by use of
  Theorem~\ref{thm:Cset-finite}\ref{i:Cset-finite}, there exists~$m
  \geq 1$ such that $\lowerc{c+1}{L_{m}}N = \lowerc{c+1}{L_{m+1}}N$.
  The same argument as used in part~\ref{i:JNNFquot-exist} shows that
  $K_{m+1} \leq L_{m+1}N$.  This shows that, amongst closed normal
  subgroups, $K$~is indeed maximal subject to $K_{n} \nleq L_{n}K$ for
  all~$n$; that is, arises as in part~\ref{i:JNNFquot-exist}.
\end{proof}

Our final result of this section is a characterization of finitely
generated Fitting-free \JNNcF\ profinite groups as inverse limits.
The natural inverse system to associate to such a group is of
virtually nilpotent profinite groups rather than of some class of
finite groups. The properties possessed by this inverse system are
analogous to those in~\cite[Theorem~4.1]{Reid-corr}.

\begin{thm}
  \label{thm:JNNF-invlim}
  Let $G$~be a finitely generated profinite group that is Fitting-free
  and let $c$~be a non-negative integer.  If $G$~is \JNNcF, then it is
  the inverse limit of a family~$G_n$, for $n \geq 0$, of profinite
  groups with respect to surjective continuous homomorphisms $\rho_{n}
  \colon G_{n+1} \to G_{n}$ with the following properties: for
  every~$n \geq 0$, \ $G_{n}$~has an open normal subgroup~$P_{n}$ such
  that, upon setting $Q_{n} = P_{n+1}\rho_{n}$,
  \begin{enumerate}
  \item \label{i:invlim-vN}
    $G_{n}$~is virtually class\nbd$c$ nilpotent;
  \item 
    $P_{n} > Q_{n}$;
  \item $\lowerc{c+1}{P_{n}} > \Mel[G_{n}]{\lowerc{c+1}{P_{n}}} \geq
    \ker\rho_{n-1} \geq \lowerc{c+1}{Q_{n}} > \1$;
  \item \label{i:invlim-dichotomy}
    if $N$~is an open normal subgroup of~$G_{n}$, then either
    $\lowerc{c+1}{N} \leq P_{n}$ or $\lowerc{c+1}{Q_{n}} \leq
    \lowerc{c+1}{N}$.
  \end{enumerate}

  Conversely, suppose, for some integer $d \geq 1$, that
  $G = \invlim G_{n}$ is an inverse limit of a countable family of
  $d$\nbd generator profinite groups with respect to surjective
  continuous homomorphisms~$\rho_{n}$ such that $G$~is Fitting-free
  and the above conditions hold.  For each~$n$, let $\pi_{n} \colon G
  \to G_{n}$~be the natural map associated to the inverse limit.  Then
  if $K$~is a non-trivial closed normal subgroup of~$G$, there
  exists~$n_{0} \geq 0$ such that $\ker\pi_{n_{0}} \leq K$.  In
  particular, $G$~is \JNNcF.
\end{thm}

In the case of finitely generated profinite groups, it is known that
\emph{any} homomorphism is necessarily continuous.  Consequently, the
word ``continuous'' could be omitted from the statement without
affecting its validity.  For arbitrary finitely generated profinite
groups, this follows by the work of Nikolov and Segal~\cite{NS1} (and
depends upon the Classification of Finite Simple Groups).  However, as
the groups~$G_{n}$ are assumed to be virtually nilpotent, it is easy
to reduce to the case of finitely generated (nilpotent) pro-$p$ groups
which was already covered by Serre; compare with~\cite{Anderson}.

\begin{proof}
  Suppose $G$~is \JNNcF\@.  Then, as observed in
  Lemma~\ref{lem:fullstruct}, there is a descending sequence $G \geq
  H_{0} > H_{1} > H_{2} > \dots$ of open normal subgroups such that
  $\bigcap_{n=0}^{\infty} H_{n} = \1$ and, for each $n \geq 1$,
  \ $\lowerc{c+1}{H_{n}} \leq \Mel[G]{\lowerc{c+1}{H_{n-1}}} <
  \lowerc{c+1}{H_{n-1}}$ and if $N \normalo G$ either $\lowerc{c+1}{N}
  \leq H_{n-1}$ or $\lowerc{c+1}{H_{n}} \leq \lowerc{c+1}{N}$.  To
  simplify notation, write $M_{n} = \Mel[G]{\lowerc{c+1}{H_{n+1}}}$
  for each~$n \geq 0$.  Then define $G_{n} = G/M_{n}$, \ $P_{n} =
  H_{n}/M_{n}$ and $Q_{n} = H_{n+1}/M_{n}$.  Let $\rho_{n} \colon
  G_{n+1} \to G_{n}$ be the natural map.  Since
  $\bigcap_{n=0}^{\infty} M_{n} \leq \bigcap_{n=0}^{\infty} H_{n} =
  \1$, it follows that $G = \invlim G_{n}$, while the conditions
  stated in the theorem all hold.  Indeed, using
  Lemma~\ref{lem:Mel}\ref{i:Mel-corr}, \ $\ker\rho_{n-1} =
  M_{n-1}/M_{n} = \Mel[G_{n}]{\lowerc{c+1}{P_{n}}}$.

  Conversely, suppose that $G = \invlim G_{n}$ is an inverse limit of
  $d$\nbd generator profinite groups~$G_n$, for $n \geq 0$, with
  respect to surjective continuous homomorphisms $\rho_{n} \colon
  G_{n+1} \to G_{n}$ such that $G$~is Fitting-free and
  conditions~\ref{i:invlim-vN}--\ref{i:invlim-dichotomy} hold where
  $P_{n} \normalo G_{n}$ and $Q_{n} = P_{n+1}\rho_{n}$.  Then $G$~is
  also $d$\nbd generated (by~\cite[Lemma~2.5.3]{RZ}).  Let $\pi_{n}
  \colon G \to G_{n}$ be the natural maps associated to the inverse
  limit.

  Observe first that the open normal subgroups $P_{1}$~and~$Q_{1}$ of
  $G_{1}$ satisfy $\lowerc{c+1}{P_{1}} > \lowerc{c+1}{Q_{1}} > \1$.
  Suppose that $G_{n}$, for some $n \geq 1$, possesses open normal
  subgroups $C_{0}$,~$C_{1}$, \dots,~$C_{n}$ such that the
  subgroups~$\lowerc{c+1}{C_{i}}$ are distinct and non-trivial.  Upon
  taking the inverse images under the homomorphism~$\rho_{n}$, we
  obtain open normal subgroups $C_{0}\rho_{n}^{\,
    -1}$,~$C_{1}\rho_{n}^{\, -1}$, \dots,~$C_{n}\rho_{n}^{\, -1}$ with
  $\lowerc{c+1}{C_{i}\rho_{n}^{-1}} \nleq \ker\rho_{n}$.  When taken
  together with~$Q_{n+1}$, these give $n+1$~open normal subgroups~$K$
  of~$G_{n+1}$ such that the corresponding~$\lowerc{c+1}{K}$ are
  distinct and non-trivial.  By induction, we conclude that $\set{
    \lowerc{c+1}{K} }{ K \normalo G_{n} }$ contains at least
  $n+1$~subgroups for all~$n$.  The corresponding set for~$G$ must
  therefore be infinite and hence $G$~is not virtually class\nbd$c$
  nilpotent by Theorem~\ref{thm:Cset-finite}\ref{i:Cset-finite}.

  Now let $K$~be a non-trivial closed normal subgroup of~$G$.  Since
  $G$~is Fitting-free, $\lowerc{c+1}{\lowerc{c+1}{K}} \neq \1$.  If
  $\lowerc{c+1}{K}\pi_{n+2} \leq P_{n+2}$ for some~$n \geq 1$, then
  $\lowerc{c+1}{\lowerc{c+1}{K}}\pi_{n} = \1$ because
  $P_{n+2}\rho_{n+1} = Q_{n+1}$ and $\lowerc{c+1}{Q_{n+1}} \leq
  \ker\rho_{n}$.  Hence there exists~$n_{0} \geq 1$ such that
  $\lowerc{c+1}{K}\pi_{n} \nleq P_{n}$ for all $n \geq n_{0}$.  Let
  $N$~be any open normal subgroup of~$G$ with $K \leq N$.  If $n \geq
  n_{0}$, then $N\pi_{n}$~is an open normal subgroup of~$G_{n}$ with
  $\lowerc{c+1}{N\pi_{n})} \nleq P_{n}$ and so $\lowerc{c+1}{Q_{n}}
  \leq \lowerc{c+1}{N\pi_{n}}$ by Condition~\ref{i:invlim-dichotomy}.
  Hence
  \[
  \lowerc{c+1}{P_{n+1}} \leq \lowerc{c+1}{N\pi_{n+1}} \ker\rho_{n}
  \leq \lowerc{c+1}{N\pi_{n+1}} \, \Mel[G_{n+1}]{\lowerc{c+1}{P_{n+1}}}
  \]
  and so we deduce $\lowerc{c+1}{P_{n+1}} \leq
  \lowerc{c+1}{N\pi_{n+1}}$ by Lemma~\ref{lem:Mel}\ref{i:Mel-incl}.
  Consequently $\ker\rho_{n} \leq \lowerc{c+1}{N\pi_{n+1}}$ for all $n
  \geq n_{0}$; that is, $\ker\pi_{n} \leq \lowerc{c+1}{N}
  \ker\pi_{n+1}$ for all $n \geq n_{0}$.  This implies
  \[
  \ker\pi_{n_{0}} \leq \bigcap_{n \geq n_{0}} \lowerc{c+1}{N}
  \ker\pi_{n} = \lowerc{c+1}{N} \leq N
  \]
  since $\lowerc{c+1}{N}$~is closed.  Now $K$~is the intersection of
  all such open normal subgroups~$N$ and therefore
  $\ker\pi_{n_{0}} \leq K$.  Consequently, $G/K$~is a quotient
  of~$G_{n_{0}}$ and so is virtually class\nbd$c$ nilpotent.  This
  demonstrates that $G$~is \JNNcF, as required.
\end{proof}


\section{Characterization of hereditarily \JNNcF\ profinite groups}
\label{sec:hJNNF}

In this section, we fix again the integer $c \geq 0$ and we shall
provide various descriptions of Fitting-free profinite groups that are
hereditarily \JNNcF.  The results that we present parallel those of
the previous section and indicate what additional properties ensure
that not only is the group itself \JNNcF, but also every open subgroup
has this property.

Let $G$~be a profinite group.  Analogous to the sets appearing in
Theorem~\ref{thm:Aset}, we define, for an open subgroup~$H$ of~$G$,
\begin{align*}
  \Astarset{H} &= \set{ \lowerc{c+1}{K} }{\text{$K \leqo G$ with $H
      \leq \Norm{G}{K}$ and $\lowerc{c+1}{K} \nleq H$}}\\
  \Cstarset{H} &= \set{ \lowerc{c+1}{K} }{\text{$K \leqc G$ with $H
      \leq \Norm{G}{K}$ and $\lowerc{c+1}{K} \nleq H$}}.
\end{align*}
If $H$~and~$L$ are open subgroups of~$G$ with $H \leq L$, we also set
$\Aset{H}(L) = \set{ \lowerc{c+1}{K} }{\text{$K \normalo L$ with
    $\lowerc{c+1}{K} \nleq H$}}$ and $\Cset{H}(L) = \set{
  \lowerc{c+1}{K} }{\text{$K \normalc L$ with $\lowerc{c+1}{K} \nleq
    H$}}$.  The following observation is straightforward:

\begin{lemma}
  \label{lem:A*-union}
  Let $G$~be a profinite group and $H$~be an open subgroup of~$G$.
  Then
  \begin{enumerate}
  \item $\Astarset{H} = \bigcup \set{ \Aset{H}(L) }{ \text{$L \leqo G$
        with $H \leq L$} }$;
  \item $\Cstarset{H} = \bigcup \set{ \Cset{H}(L) }{ \text{$L \leqo G$
      with $H \leq L$} }$.
  \end{enumerate}
\end{lemma}


In order to establish Theorem~\ref{thm:Astarset}, which is the
analogue of Theorem~\ref{thm:Aset} for hereditarily \JNNcF\ groups, we
need to know that the condition that the group is Fitting-free is
inherited by open subgroups.  We establish this in
Lemma~\ref{lem:open-noabelian} below.  We shall use the following
analogue of an observation made in the proof
of~\cite[(2.1)]{JSW-Large}.  The argument is similar but included for
completeness.

\begin{lemma}
  \label{lem:JNNF-infconjs}
  Let $G$~be a Fitting-free \JNNcF\ profinite group.  Then
  \begin{enumerate}
  \item \label{i:JNNF-eltconj}
    every non-identity element of~$G$ has infinitely many conjugates
    in~$G$;
  \item \label{i:subgp-infconjs}
    if $H$~is a non-trivial finite subgroup of~$G$, then $H$~has
    infinitely many conjugates in~$G$.
  \end{enumerate}
\end{lemma}

\begin{proof}
  \ref{i:JNNF-eltconj}~Suppose that $x$~is a non-identity element
  of~$G$ with finitely many conjugates in~$G$.  Let $X$~be the closed
  normal subgroup of~$G$ generated by the conjugates of~$x$ and $C$~be
  the intersection of the centralizers in~$G$ of each conjugate
  of~$x$.  Since $x$~has finitely many conjugates, $C$~is open in~$G$
  and, in particular, non-trivial.  Since $[C,X] = \1$, it follows
  that $C \cap X$~is an abelian closed normal subgroup of~$G$ and so
  $C \cap X = \1$ by assumption.  This contradicts
  Lemma~\ref{lem:JNNF-intersect}.

  \ref{i:subgp-infconjs}~Let $H$~be a non-trivial finite subgroup
  of~$G$ with finitely many conjugates in~$G$.  If $x$~is a
  non-identity element of~$H$, then every conjugate of~$x$ belongs to
  one of the conjugates of~$H$.  It follows that $x$~has finitely many
  conjugates in~$G$, which contradicts~\ref{i:JNNF-eltconj}.
\end{proof}

\begin{lemma}
  \label{lem:open-noabelian}
  Let $G$~be a Fitting-free \JNNcF\ profinite group.  If $H$~is any
  open subgroup of~$G$, then $H$~is also Fitting-free.
\end{lemma}

\begin{proof}
  Suppose that $A$~is an abelian closed normal subgroup of~$H$.  Let
  $B = A \cap \Core{G}{H}$.  Note that $B$~has finitely many
  conjugates in~$G$ and each of them is a normal subgroup
  of~$\Core{G}{H}$.  Hence the normal closure~$B^{G}$ is the product
  of these subgroups and this is nilpotent by Fitting's Theorem.
  Since $G$~is Fitting-free, it follows that $B = \1$.  Therefore
  $A$~is finite and so, by
  Lemma~\ref{lem:JNNF-infconjs}\ref{i:subgp-infconjs}, $A = \1$.
\end{proof}

\begin{thm}
  \label{thm:Astarset}
  Let $G$~be a finitely generated infinite profinite group that is
  Fitting-free and let $c$~be a non-negative integer.  Then the
  following conditions are equivalent:
  \begin{enumerate}
  \item \label{i:HJNNF}
    $G$~is hereditarily \JNNcF;
  \item \label{i:A*-finite}
    the set~$\Astarset{H}$ is finite for every open subgroup~$H$
    of~$G$;
  \item \label{i:C*-finite}
    the set $\Cstarset{H}$ is finite for every open subgroup~$H$
    of~$G$.
  \end{enumerate}
\end{thm}

\begin{proof}
  Suppose first that $\Astarset{H}$~is finite for every open
  subgroup~$H$ of~$G$.  Since $\Aset{H} \subseteq \Astarset{H}$, it
  follows that $G$~is \JNNcF\ by Theorem~\ref{thm:Aset}.  Let $L$~be
  an open subgroup of~$G$.  Then $L$~is Fitting-free by
  Lemma~\ref{lem:open-noabelian} and $\Aset{H}(L)$~is finite for every
  open subgroup~$H$ of~$L$ as it is contained~$\Astarset{H}$.  Hence
  $L$~is also \JNNcF\ by Theorem~\ref{thm:Aset}.  This establishes
  $\text{\ref{i:A*-finite}} \Rightarrow \text{\ref{i:HJNNF}}$.

  Since $\Astarset{H} \subseteq \Cstarset{H}$, certainly
  $\text{\ref{i:C*-finite}} \Rightarrow \text{\ref{i:A*-finite}}$.
  Finally assume that $G$~is hereditarily \JNNcF\ and let $H$~be an
  open subgroup of~$G$.  There are finitely many open subgroups~$L$
  of~$G$ with $H \leq L$.  If $L$~is such an open subgroup, then
  $L$~is \JNNcF, so $\Cset{H}(L)$~is finite by Theorem~\ref{thm:Aset}
  together with Lemma~\ref{lem:open-noabelian}.  Hence
  $\Cstarset{H}$~is a finite union of finite sets, by
  Lemma~\ref{lem:A*-union}, and so is finite.  This establishes the
  final implication $\text{\ref{i:HJNNF}} \Rightarrow
  \text{\ref{i:C*-finite}}$.
\end{proof}

Wilson~\cite[(2.1)]{JSW-Large} characterizes when a just infinite
group is not hereditarily just infinite.  The following is our
analogue for \JNNcF\ groups.  The same method is used to construct the
basal subgroup~$K$ but then a few additional steps establish its
properties.

\begin{prop}
  \label{prop:notHJNNF}
  Let $G$~be a Fitting-free \JNNcF\ profinite group that is not
  hereditarily \JNNcF\@.  Then $G$~has an infinite closed basal
  subgroup~$K$ such that $\Norm{G}{K}/K$~is not virtually class\nbd$c$
  nilpotent and $K$~has no non-trivial abelian closed subgroup that is
  topologically characteristic in~$K$.  In particular, $K$~is not
  normal in~$G$.
\end{prop}

\begin{proof}
  Since $G$~is not hereditarily \JNNcF, there is an open subgroup~$H$
  of~$G$ and a non-trivial closed normal subgroup~$L$ of~$H$ such that
  $H/L$~is not virtually class\nbd$c$ nilpotent.  Let $C$~be the core
  of~$H$ in~$G$.  If $C \cap L = \1$, then $L$~is finite, which is a
  contradiction by
  Lemma~\ref{lem:JNNF-infconjs}\ref{i:subgp-infconjs}. Hence $C \cap L
  \neq \1$.  Note that $CL/L$~is a subgroup of finite index in~$H/L$
  and is isomorphic to~$C/(C \cap L)$.  If $C/(C \cap L)$~were
  virtually class\nbd$c$ nilpotent, then as $H/L$~is a finite
  extension we would obtain another contradiction.  Hence $C/(C \cap
  L)$~is not virtually class\nbd$c$ nilpotent and we may replace
  $H$~and~$L$ by $C$~and~$C \cap L$, respectively, and assume that
  $H$~is an open \emph{normal} subgroup of~$G$ with a non-trivial
  closed normal subgroup~$L$ such that $H/L$~is not virtually
  class\nbd$c$ nilpotent.

  Now $L$~has finitely many conjugates in~$G$ and these are all
  contained in~$H$.  Hence $L \normal L^{G}$ and
  Lemma~\ref{lem:Reid-basal} tells us that we can construct a basal
  subgroup~$K$ of~$G$ by intersecting a suitable collection of the
  conjugates of~$L$.  We may assume that $L$~is one of these
  conjugates so that $K \leq L$.  Note that $K$~is infinite by use of
  Lemma~\ref{lem:JNNF-infconjs}\ref{i:subgp-infconjs}.  If
  $\Norm{G}{K}/K$~were virtually class\nbd$c$ nilpotent, then so would
  be~$H/L$ since $K \leq L \leq H \leq \Norm{G}{K}$, contrary to our
  hypothesis.  Note then that $K$~is cannot be normal in~$G$ since if
  it were then $\Norm{G}{K}/K = G/K$ would be virtually class\nbd$c$
  nilpotent.  Finally if $A$~were a non-trivial abelian closed
  subgroup that is topologically characteristic in~$K$, then as
  conjugation is a homeomorphism there would be precisely one $G$\nbd
  conjugate of~$A$ in each conjugate of~$K$.  Hence $A$~would also be
  basal and its normal closure~$A^{G}$ would be a non-trivial abelian
  normal subgroup of~$G$, contrary to assumption.  This establishes
  the claimed conditions.
\end{proof}

Using the characterization given in Theorem~\ref{thm:Astarset}, we are
able to give a description of hereditarily \JNNcF\ groups of a similar
form to our earlier Theorem~\ref{thm:JNNFstruct}.

\begin{thm}
  Let $G$~be a finitely generated profinite group that is Fitting-free
  and let $c$~be a non-negative integer.  Then $G$~is hereditarily
  \JNNcF\ if and only if there is a descending sequence $G \geq H_{0}
  > H_{1} > H_{2} > \dots$ of open normal subgroups such that
  \begin{enumerate}
  \item \label{i:Hstruct-intersect}
    $\bigcap_{n=0}^{\infty} H_{n} = \1$, and
  \item \label{i:Hstruct-dichot}
    if $L$~is an open subgroup of~$G$ that is normalized
    by~$H_{n-1}$ for some $n \geq 1$, then either $\lowerc{c+1}{L}
    \leq H_{n-1}$ or $\lowerc{c+1}{H_{n}} \leq \lowerc{c+1}{L}$.
  \end{enumerate}
\end{thm}

\begin{proof}
  Suppose first that $G$~is hereditarily \JNNcF\@.  We start with any
  open normal subgroup~$H_{0}$ of~$G$.  Suppose then, as an inductive
  hypothesis, that we have constructed open normal subgroups
  $G \geq H_{0} > H_{1} > \dots > H_{n-1}$ such that, for each $i \in
  \{1, \ldots, n-1\}$, \ $\lowerc{c+1}{H_{i-1}} > \lowerc{c+1}{H_{i}}$
  and if $L$~is normalized by~$H_{i-1}$ then either $\lowerc{c+1}{L}
  \leq H_{i-1}$ or $\lowerc{c+1}{H_{i}} \leq \lowerc{c+1}{L}$.  By
  Theorem~\ref{thm:Astarset}, the set
  \[
  \Astarset{H_{n-1}} = \set{ \lowerc{c+1}{K} }{\text{$K \leqo G$ with
      $H_{n-1} \leq \Norm{G}{K}$ and $\lowerc{c+1}{K} \nleq H_{n-1}$}}
  \]
  contains finitely many members.  Use of Lemma~\ref{lem:JNNF-Mel}
  shows that $\Mel[G]{\lowerc{c+1}{H_{n-1}}} \neq \1$.  Hence
  \[
  R = \Mel[G]{\lowerc{c+1}{H_{n-1}}} \cap \bigcap \Astarset{H_{n-1}}
  \]
  is a non-trivial closed normal subgroup of~$G$ (by
  Lemma~\ref{lem:JNNF-intersect}).  The quotient~$G/R$ is then
  virtually class\nbd$c$ nilpotent, so there exists an open normal
  subgroup~$S$ with $\lowerc{c+1}{S} \leq R$.  Take $H_{n} = H_{n-1}
  \cap S$, so that $H_{n}$~is an open normal subgroup of~$G$ contained
  in~$H_{n-1}$ with $\lowerc{c+1}{H_{n}} \leq R <
  \lowerc{c+1}{H_{n-1}}$.  If $L$~is an open subgroup normalized
  by~$H_{n-1}$, then either $\lowerc{c+1}{L} \leq H_{n-1}$ or
  $\lowerc{c+1}{L} \in \Astarset{H_{n-1}}$.  In the latter case,
  $\lowerc{c+1}{H_{n}} \leq R \leq \lowerc{c+1}{L}$.

  Repeating this process constructs a descending sequence of open
  normal subgroups~$H_{n}$ such that condition~\ref{i:Hstruct-dichot}
  holds.  If the intersection $J = \bigcap_{n=0}^{\infty} H_{n}$ were
  non-trivial, then $G/\lowerc{c+1}{J}$~would be virtually
  class\nbd$c$ nilpotent, but would have infinitely many distinct
  subgroups~$\lowerc{c+1}{H_{n}/\lowerc{c+1}{J}}$ contrary to
  Theorem~\ref{thm:Cset-finite}.  Hence
  condition~\ref{i:Hstruct-intersect} also holds.

  Conversely suppose that $G$~is a finitely generated profinite group
  that has no non-trivial abelian closed normal subgroup with a
  descending sequence of open normal subgroups~$H_{n}$ satisfying
  conditions \ref{i:Hstruct-intersect} and~\ref{i:Hstruct-dichot}.  In
  particular, $G$~satisfies the conditions appearing in
  Theorem~\ref{thm:JNNFstruct} and so is \JNNcF\@.  Suppose that it is
  not hereditarily \JNNcF\@.  By Proposition~\ref{prop:notHJNNF},
  $G$~has a closed basal subgroup~$K$ with no non-trivial abelian
  topologically characteristic subgroup such that $\Norm{G}{K}/K$~is
  not virtually class\nbd$c$ nilpotent.  Then $\lowerc{c+1}{K} \neq
  \1$, so there exists~$m \geq 0$ such that $\lowerc{c+1}{K} \nleq
  H_{m}$.  Since $\bigcap_{n=0}^{\infty} H_{n} = \1$, it follows that
  every open subgroup of~$G$ contains some~$H_{n}$.  Hence, by
  increasing~$m$ if necessary, we can assume $H_{m} \leq \Norm{G}{K}$.
  Let $U$~be any open normal subgroup of~$G$ and $L = KU$.  Then
  $H_{m}$~normalizes~$L$ and $\lowerc{c+1}{L} \nleq H_{m}$.  Hence, by
  condition~\ref{i:Hstruct-dichot},
  $\lowerc{c+1}{H_{m+1}} \leq \lowerc{c+1}{L}$.  It follows that
  \[
  \lowerc{c+1}{H_{m+1}} \leq \bigcap_{U \normalo G} KU = K.
  \]
  Therefore $\Norm{G}{K}/K$~is isomorphic to a quotient of a subgroup
  of~$G/\lowerc{c+1}{H_{m+1}}$ and hence is virtually class\nbd$c$
  nilpotent.  This is a contradiction and we conclude that $G$~is
  indeed hereditarily \JNNcF, as claimed.
\end{proof}

We complete the section by giving a suitable description of a
hereditarily \JNNcF\ profinite group as an inverse limit of virtually
nilpotent groups in a manner analogous to the description appearing in
Theorem~\ref{thm:JNNF-invlim}.

\begin{thm}
  \label{thm:Hinvlim}
  Let $G$~be a finitely generated profinite group that is Fitting-free
  and let $c$~be a non-negative integer.  If $G$~is hereditarily
  \JNNcF, then it is the inverse limit of a family~$G_{n}$, for $n
  \geq 0$, of profinite groups with respect to surjective continuous
  homomorphisms $\rho_{n} \colon G_{n+1} \to G_{n}$ with the following
  properties: for every~$n \geq 0$, \ $G_{n}$~has an open normal
  subgroup~$P_{n}$ such that, upon setting $Q_{n} = P_{n+1}\rho_{n}$,
  \begin{enumerate}
  \item \label{i:Hinvlim-vN}
    $G_{n}$~is virtually class\nbd$c$ nilpotent;
  \item \label{i:Hinvlim-openinc}
    $P_{n} > Q_{n}$;
  \item \label{i:Hinvlim-derivs}
    $\lowerc{c+1}{P_{n}} > \Mel[G_{n}]{\lowerc{c+1}{P_{n}}} \geq
    \ker\rho_{n-1} \geq \lowerc{c+1}{Q_{n}} > \1$;
  \item \label{i:Hinvlim-dichot}
    if $N$~is an open normal subgroup of~$G_{n}$, then either
    $\lowerc{c+1}{N} \leq P_{n}$ or $\lowerc{c+1}{Q_{n}} \leq
    \lowerc{c+1}{N}$;
  \item \label{i:Hinvlim-basal}
    there is no non-normal closed subgroup~$V$ of~$G_{n}$ with at most
    $n$~conjugates such that any pair of distinct conjugates of~$V$
    centralize each other and such that the normal closure~$W = V^{G}$
    satisfies $\lowerc{c+1}{P_{n}} \leq \lowerc{c+1}{\lowerc{c+1}{W}}$.
  \end{enumerate}

  Conversely, if, for some integer $d \geq 1$, $G = \invlim G_{n}$ is
  an inverse limit of a countable family of $d$\nbd generator
  profinite groups with respect to surjective continuous
  homomorphisms~$\rho_{n}$ such that $G$~is Fitting-free and the above
  conditions hold, then $G$~is hereditarily \JNNcF.
\end{thm}

\begin{proof}
  Suppose that $G$~is hereditarily \JNNcF\@.  Since $G$~is finitely
  generated it has finitely many open subgroups of each index and so
  we can enumerate a sequence of open normal subgroups~$U_{n}$ of~$G$
  such that, for each $n \geq 1$, every open subgroup of index at
  most~$n$ contains~$U_{n}$.  Take $H_{0}$~to be any open normal
  subgroup of~$G$.  Certainly $\lowerc{c+1}{H_{0}} \neq \1$.  Now
  assume, as an inductive hypothesis, that we have constructed a
  sequence of open normal subgroups $G \geq H_{0} > H_{1} > \dots >
  H_{n-1}$.  By Theorem~\ref{thm:Astarset}, the
  set~$\Astarset{H_{n-1}}$ is finite while
  $\Mel[G]{\lowerc{c+1}{H_{n-1}}}$~is non-trivial by
  Lemma~\ref{lem:JNNF-Mel}.  Hence, by Lemma~\ref{lem:JNNF-intersect},
  \[
  R = \Mel[G]{\lowerc{c+1}{H_{n-1}}} \cap \left( \bigcap
  \Astarset{H_{n-1}} \right)'
  \]
  is a non-trivial closed subgroup of~$G$, so $G/R$~is virtually
  class\nbd$c$ nilpotent and there exists an open normal subgroup~$S$
  of~$G$ with $\lowerc{c+1}{S} \leq R$.  Take $H_{n} = H_{n-1} \cap
  U_{n} \cap S$.  In particular, $\lowerc{c+1}{H_{n}} \leq R \leq
  \Mel[G]{\lowerc{c+1}{H_{n-1}}} < \lowerc{c+1}{H_{n-1}}$.  By
  repeated application, we conclude there is a descending sequence of
  open normal subgroups $G \geq H_{0} > H_{1} > H_{2} > \dots$ such
  that $H_{n} \leq U_{n}$ and
  \[
  \lowerc{c+1}{H_{n}} \leq \Mel[G]{\lowerc{c+1}{H_{n-1}}} \cap
  \left( \bigcap \Astarset{H_{n-1}} \right)' <
  \lowerc{c+1}{H_{n-1}}
  \]
  for all $n \geq 1$.  Since $H_{n} \leq U_{n}$ for each~$n$, it
  immediately follows that $\bigcap_{n=0}^{\infty} H_{n} = \1$.

  Now, for $n \geq 0$, write $M_{n} = \Mel[G]{\lowerc{c+1}{H_{2n+2}}}$
  and define $G_{n} = G/M_{n}$, $P_{n} = H_{2n}/M_{n}$ and
  $Q_{n} = H_{2n+2}/M_{n}$.  Let $\rho_{n} \colon G_{n+1} \to G_{n}$
  be the natural map.  Since $\bigcap_{n=0}^{\infty} M_{n} = \1$, it
  is the case that $G = \invlim G_{n}$.  Since each~$M_{n} \neq \1$,
  the assumption that $G$~is \JNNcF\ ensures each~$G_{n}$ is virtually
  class\nbd$c$ nilpotent and Conditions
  \ref{i:Hinvlim-openinc}~and~\ref{i:Hinvlim-derivs} follow
  immediately from the construction.  Indeed $\ker\rho_{n-1} =
  M_{n-1}/M_{n} = \Mel[G_{n}]{\lowerc{c+1}{P_{n}}}$ using
  Lemma~\ref{lem:Mel}\ref{i:Mel-corr}.  If $N \normalo G_{n}$, say $N
  = K/M_{n}$, with $\lowerc{c+1}{N} \nleq P_{n}$, then
  $\lowerc{c+1}{K} \in \Aset{H_{2n}} \subseteq \Astarset{H_{2n}}$.
  Hence $\lowerc{c+1}{H_{2n+2}} < \lowerc{c+1}{H_{2n+1}} \leq \bigcap
  \Astarset{H_{2n}} \leq \lowerc{c+1}{K}$ and this establishes
  Condition~\ref{i:Hinvlim-dichot}.

  Suppose there is a non-normal closed subgroup~$V$ of~$G_{n}$ with at
  most $n$~conjugates such that $[V^{g},V^{h}] = \1$ when $gh^{-1}
  \notin \Norm{G_{n}}{V}$ and such that the normal closure $W =
  V^{G_{n}}$ satisfies $\lowerc{c+1}{P_{n}} \leq
  \lowerc{c+1}{\lowerc{c+1}{W}}$.  Since elements from distinct
  conjugates of~$V$ commute, $\lowerc{c+1}{\lowerc{c+1}{W}}$~is the
  product of the conjugates of~$\lowerc{c+1}{\lowerc{c+1}{V}}$.  Write
  $V = K/M_{n}$ and $W = L/M_{n}$.  Then observe $L = K^{G}$,
  $\lowerc{c+1}{\lowerc{c+1}{L}} \leq
  \lowerc{c+1}{\lowerc{c+1}{K}}^{G} M_{n}$ and $\lowerc{c+1}{H_{2n}}
  \leq \lowerc{c+1}{\lowerc{c+1}{L}} M_{n}$, which implies
  $\lowerc{c+1}{H_{2n}} \leq \lowerc{c+1}{\lowerc{c+1}{L}}$ with use
  of Lemma~\ref{lem:Mel}\ref{i:Mel-incl}.  Also $K$~has at most
  $n$~conjugates in~$G$, so it must be the case that $H_{2n+1} \leq
  U_{n} \leq \Norm{G}{K}$.  Now $\lowerc{c+1}{H_{2n+1}} <
  \lowerc{c+1}{\lowerc{c+1}{L}}$, so $\lowerc{c+1}{\lowerc{c+1}{K}}
  \nleq \lowerc{c+1}{H_{2n+1}}$ and therefore $\lowerc{c+1}{K} \nleq
  H_{2n+1}$.  In conclusion, for each~$i \geq 0$, \ $KH_{i}$~is an
  open subgroup of~$G$ with the property that $\lowerc{c+1}{KH_{i}}
  \in \Astarset{H_{2n+1}}$.  Thus
  \[
  \bigcap \Astarset{H_{2n+1}} \leq
  \bigcap_{i \geq 0} \lowerc{c+1}{K} H_{i} = \lowerc{c+1}{K}.
  \]
  Since $\bigcap \Astarset{H_{2n+1}}$~is a normal subgroup, it is
  contained in all conjugates of~$K$ and therefore
  \[
  \lowerc{c+1}{H_{2n+2}} \leq
  \left( \bigcap \Astarset{H_{2n+1}} \right)' \leq 
  [K^{g},K^{h}]
  \]
  for all $g,h \in G$.  Consequently, $\1 \neq \lowerc{c+1}{Q_{n}}
  \leq [V^{g},V^{h}]$ for all $g,h \in G_{n}$.  However, as $V$~is not
  normal in~$G_{n}$ there exists $g,h \in G_{n}$ such that
  $V^{g}$~and~$V^{h}$ are distinct and these satisfy $[V^{g},V^{h}] =
  \1$.  This contradiction establishes
  Condition~\ref{i:Hinvlim-basal}.

  \spc

  Conversely, suppose that $G = \invlim G_{n}$ is an inverse limit of
  $d$\nbd generator profinite groups~$G_n$, for $n \geq 0$, with
  respect to surjective continuous homomorphisms $\rho_{n} \colon
  G_{n+1} \to G_{n}$ such that $G$~has no non-trivial abelian closed
  normal subgroup and that conditions
  \ref{i:Hinvlim-vN}--\ref{i:Hinvlim-basal} hold where $P_{n} \normalo
  G_{n}$ and $Q_{n} = P_{n+1}\rho_{n}$.  In particular, the conditions
  of Theorem~\ref{thm:JNNF-invlim} are satisfied and so $G$~is
  \JNNcF\@.  Let $\pi_{n} \colon G \to G_{n}$ be the natural maps
  associated to the inverse limit.  Suppose that $G$~is not
  hereditarily \JNNcF\@.  Then by Proposition~\ref{prop:notHJNNF},
  $G$~has some closed non-normal basal subgroup~$K$.  Take~$n_{0}$ to
  be a positive integer such that $K$~has fewer than~$n_{0}$
  conjugates in~$G$ and set $L = K^{G}$, the direct product of the
  conjugates of~$K$.

  Since $\lowerc{c+1}{\lowerc{c+2}{L}} \neq \1$, it is the case that
  $\ker\pi_{n} \leq \lowerc{c+1}{\lowerc{c+2}{L}}$ for all
  sufficiently large~$n$ by Theorem~\ref{thm:JNNF-invlim}.  Hence,
  increasing~$n_{0}$ if necessary, we may assume that $\ker\pi_{n} <
  \lowerc{c+1}{\lowerc{c+2}{L}} \leq L'$ for all $n \geq n_{0}$.  The
  subgroup~$K$ has at least two conjugates in~$G$ and any distinct
  pair commutes as $K$~is basal.  If $K\pi_{n}$~were normal
  in~$G_{n}$, then the images of these conjugates would coincide and
  so $L\pi_{n} = K\pi_{n}$ would be abelian.  This is impossible since
  $\ker\pi_{n} < L'$.  Since the number of conjugates cannot increase
  in the image, we deduce that, when $n \geq n_{0}$, \ $K\pi_{n}$~is a
  closed subgroup of~$G_{n}$ that is not normal and has at most
  $n_{0}$~conjugates in~$G_{n}$. For such~$n$, if $x \in
  \lowerc{c+1}{P_{n+2}}$, write $x = g\pi_{n+2}$ for some~$g \in G$.
  Using the fact that $\lowerc{c+1}{Q_{n+1}} \leq \ker\rho_{n}$, one
  observes $g \in \ker\pi_{n} \leq \lowerc{c+1}{\lowerc{c+1}{L}}$ and
  therefore $\lowerc{c+1}{P_{n+2}} \leq
  \lowerc{c+1}{\lowerc{c+1}{L\pi_{n+2}}}$ for $n \geq n_{0}$.  In
  particular, for such~$n$, taking $V = K\pi_{n+2}$ and $W =
  L\pi_{n+2}$ in~$G_{n+2}$ contradicts the hypothesis in
  Condition~\ref{i:Hinvlim-basal}.
\end{proof}

When comparing the above description of hereditarily \JNNcF\ groups
with the corresponding result of Reid~\cite[Theorem~5.2]{Reid-corr}
for hereditarily just infinite groups, one notices the bound on the
number of conjugates appearing in our Condition~\ref{i:Hinvlim-basal}.
There seems to be no analogue in the corresponding description of
hereditarily just infinite groups.  However, note that the bound
of~$n$ for the number of conjugates could, with only minor adjustment
to the proof, be replaced by some bound~$f(n)$ where $f \colon \Nat
\to \Nat$ is any strictly increasing function.  In~\cite{Reid-corr},
the hereditarily just infinite group is isomorphic to an inverse limit
$G = \invlim G_{n}$ of \emph{finite} groups and there is therefore an
implicit bound on the number of conjugates for subgroups of~$G_{n}$.
Consequently, this condition is quite reasonable.


\section{Subgroups of finite index in \JNNcF\ groups}
\label{sec:finiteindex}

In this section we shall establish Theorem~\ref{thm:Hmax} (see
Corollary~\ref{cor:max}) and so consider both profinite groups and
discrete groups.  We shall adopt the common convention that, in the
case of profinite groups, all subgroups are assumed to be within the
same category and so ``subgroup'' means ``closed subgroup'' in this
case.  This enables our results to be more streamlined in their
statement and the proofs correspondingly cleaner.  We fix the integer
$c \geq 0$ throughout and begin with an observation that is, modulo
our standard assumption about abelian normal subgroups, an improvement
on Corollary~\ref{cor:nofinite}.

\begin{lemma}
  \label{lem:noVN}
  Let $G$~be a profinite group or discrete group that is \JNNcF\ and
  Fitting-free.  Then $G$~has no non-trivial normal subgroup that is
  virtually nilpotent.
\end{lemma}

\begin{proof}
  Suppose that $N$~is a non-trivial normal subgroup of~$G$ with a
  nilpotent normal subgroup of finite index in~$N$.  The Fitting
  subgroup~$\Fitt{N}$ of~$N$ is then a product of finitely many
  nilpotent normal subgroups of~$N$ and so is a nilpotent normal
  subgroup of~$G$.  Since $G$~is Fitting-free, it follows that $N$~is
  finite.  Then $\Cent{G}{N}$~has finite index in~$G$, which
  contradicts Lemma~\ref{lem:FittingFree-Cent}.
\end{proof}

\begin{lemma}
  \label{lem:HJNNF-normalredn}
  Let $G$~be a profinite group or a discrete group that is
  Fitting-free.  Suppose that every normal subgroup of finite index is
  \JNNcF\@.  Then $G$~is hereditarily \JNNcF.
\end{lemma}

\begin{proof}
  Suppose that $H$~is a subgroup of finite index~$G$ and that $N$~is a
  non-trivial normal subgroup of~$H$.  Let $K = \Core{G}{H}$, so that
  $K$~is a normal subgroup of~$G$ also of finite index and hence
  \JNNcF\ by hypothesis.  If it were the case that $K \cap N = \1$,
  then $[K,N] = \1$ since both $K$~and~$N$ are normal subgroups
  of~$H$.  Then $N \leq \Cent{G}{K}$, in contradiction to
  Lemma~\ref{lem:FittingFree-Cent}.  We deduce therefore that $K \cap
  N \neq \1$.  Then $H/N$~is a finite extension of $KN/N \cong K/(K
  \cap N)$, which is virtually class\nbd$c$ nilpotent.  Hence $H$~is
  \JNNcF, as required.
\end{proof}

Recall that the \emph{finite radical}~$\Fin{G}$ of a group~$G$ is the
union of all finite normal subgroups of~$G$.  The following is a
\JNNcF\ analogue of~\cite[Lemma~4]{Reid10a}.

\begin{lemma}
  \label{lem:Fin}
  \begin{enumerate}
  \item \label{i:Fin-finiteindex}
    Let $G$~be a group with $\Fin{G} = \1$.  If $H$~is a subgroup of
    finite index, then $\Fin{H} = \1$.
  \item \label{i:Fin-JNNF}
    Let $G$~be a profinite or discrete group with $\Fin{G} = \1$ and
    $H$~be a subgroup of finite index that is \JNNcF\@.  Then every
    subgroup of~$G$ containing~$H$ is \JNNcF\@.
  \end{enumerate}
\end{lemma}

\begin{proof}
  \ref{i:Fin-finiteindex}~This is established
  in~\cite[Lemma~4]{Reid10a}.
  
  \ref{i:Fin-JNNF}~Suppose that $H \leq L \leq G$.  First note that
  $L$~is not virtually class\nbd$c$ nilpotent as it contains~$H$.  Let
  $K$~be a non-trivial normal subgroup of~$L$.  Since $\Fin{L} = \1$
  by part~\ref{i:Fin-finiteindex}, $K$~is infinite.  As $H \cap K$~has
  finite index in~$K$, it follows that $H \cap K$~is non-trivial and
  so $H/(H \cap K)$~is virtually class\nbd$c$ nilpotent.  We conclude
  that $L/K$~is a finite extension of~$HK/K \cong H/(H \cap K)$, so
  $L/K$~is virtually class\nbd$c$ nilpotent.  Hence $L$~is \JNNcF\@.
\end{proof}

We are now in a position to establish a theorem for \JNNcF\ groups
that is an analogue of the main theorem of~\cite{Reid10a}:

\begin{thm}
  \label{thm:JNNF-finiteindexnormal}
  Let $G$~be a profinite group or a discrete group and let $c$~be a
  non-negative integer.  Suppose that $G$~is \JNNcF\ and Fitting-free,
  and that $H$~is a normal subgroup of finite index in~$G$.  Then the
  following are equivalent:
  \begin{enumerate}
  \item \label{i:H-JNNF}
    $H$~is \JNNcF;
  \item \label{i:containH-JNNF}
    every subgroup of~$G$ containing~$H$ is \JNNcF;
  \item \label{i:max-JNNF}
    every maximal subgroup of~$G$ containing~$H$ is \JNNcF.
  \end{enumerate}
\end{thm}

\begin{proof}
  By Lemma~\ref{lem:noVN}, $\Fin{G} = \1$.  Hence an application of
  Lemma~\ref{lem:Fin}\ref{i:Fin-JNNF} shows that
  Condition~\ref{i:H-JNNF} implies Condition~\ref{i:containH-JNNF}.
  It is trivial that Condition~\ref{i:containH-JNNF} implies
  Condition~\ref{i:max-JNNF}.

  Now assume Condition~\ref{i:max-JNNF}.  Let $K$~be a non-trivial
  normal subgroup of~$H$.  Since $H$~is a normal subgroup of~$G$, we
  observe that $K^{g} \normal H \leq \Norm{G}{K}$ for all $g \in G$
  and hence $K \normal K^{G}$.  By Lemma~\ref{lem:Reid-basal}, there
  is a basal subgroup~$B$ that is an intersection of some conjugates
  of~$K$ and, conjugating if necessary, we may assume $B \leq K$.
  Note also that $H \leq \Norm{G}{B}$ since each conjugate of~$K$ is
  normal in~$H$.  We shall show that $B$~is normal in~$G$.  For then,
  $G/B$~is virtually class\nbd$c$ nilpotent by hypothesis and hence
  $H/K$~is also virtually class\nbd$c$ nilpotent since $B \leq K$.
  This will establish that $H$~is indeed \JNNcF\@.
  
  Suppose, for a contradiction, that $B$~is not a normal subgroup
  of~$G$.  Consequently, $\Norm{G}{B}$~is a proper subgroup of~$G$ and
  there is some maximal subgroup~$M$ of~$G$ with $\Norm{G}{B} \leq M$.
  Now $B^{G}$~is the direct product of the conjugates of~$B$ and it is
  not virtually nilpotent by Lemma~\ref{lem:noVN}.  Observe that
  $B$~has fewer conjugates in~$M$ than in the group~$G$, so
  $B^{G}/B^{M}$~is isomorphic to a direct product of some copies
  of~$B$ and so is not virtually nilpotent.  On the other hand, $M$~is
  \JNNcF\ by assumption, so the quotient~$M/B^{M}$ of~$M$ by the
  normal closure of~$B$ in~$M$ is a virtually nilpotent group.  Hence
  $(M \cap B^{G})/B^{M}$~is virtually nilpotent and this implies
  $B^{G}/B^{M}$~is also virtually nilpotent since $M \cap B^{G}$~has
  finite index in~$B^{G}$.  This is a contradiction and completes the
  proof of the theorem.
\end{proof}

With use of Lemma~\ref{lem:HJNNF-normalredn}, we then immediately
conclude:

\begin{cor}
  \label{cor:max}
  Let $G$~be a profinite or discrete group that is \JNNcF\ and
  Fitting-free.  Then $G$~is hereditarily \JNNcF\ if and only if every
  maximal (open) subgroup of finite index is \JNNcF.
\end{cor}


\section{A construction of hereditarily \JNNcF\ groups}
\label{sec:construction}

The work of the preceding sections suggest that \JNNcF\ groups are
quite closely related to just infinite groups.  Similarly, Wilson's
classification~\cite{JSW71,JSW00} of just infinite groups has the same
dichotomy as Hardy's~\cite{Hardy-PhD} for JNAF groups, namely branch
groups and subgroups of wreath products built from hereditarily just
infinite, or respectively JNAF, groups.  To fully investigate the
class of \JNNcF\ groups, one would like a good supply of examples of
hereditarily \JNNcF\ groups.  In this section, we provide one method
for constructing such a group in Theorem~\ref{thm:JI-automs}.  At
first sight the construction may appear somewhat unspectacular since
it merely consists of a semidirect product of a hereditarily just
infinite group~$H$ by some group~$A$ of (outer) automorphisms.
However, by applying it to a variety of known hereditarily just
infinite groups~$H$ and observing that the range of possible
groups~$A$ that could be used is rather wide, we manufacture
interesting examples of \JNNcF\ groups.  In both
Examples~\ref{ex:JNAF1} and~\ref{ex:JNAF2}, we shall observe that,
with suitable choices of ingredients for $H$, then among abelian
profinite groups the options for~$A$ are about as wide as could be
hoped for.  For example, one can take~$A$ to be any closed subgroup of
the Cartesian product of countably many copies of the profinite
completion~$\hat{\Zint}$ of the integers.  In Example~\ref{ex:Nott},
we are able to take~$A$ to be any finitely generated virtually
nilpotent pro\nbd$p$ group and so again this permits a wide range of
possible choices.

\begin{lemma}
  \label{lem:noInn-centralizer}
  Let $H$~be a group and $A$~be a group of automorphisms of~$H$ such
  that $A \cap \Inn{H} = \1$.  Define $G = H \rtimes A$ to be the
  semidirect product of~$H$ by~$A$ via its natural action on~$H$.
  Then $\Cent{G}{H} = \Centre{H}$.
\end{lemma}

\begin{proof}
  Let $x = h\alpha \in \Cent{G}{H}$ with $h \in H$ and $\alpha \in A$.
  If $\tau_{h}$~denotes the inner automorphism of~$H$ induced by~$h$
  on~$H$, then we observe $\tau_{h}\alpha = 1$ in~$\Aut H$, so $\alpha
  \in \Inn{H}$.  Hence $\alpha = 1$, so $x = h$ and necessarily $h \in
  \Centre{H}$.  The reverse inclusion is trivial.
\end{proof}

\begin{thm}
  \label{thm:JI-automs}
  Let $H$~be a hereditarily just infinite (discrete or profinite)
  group that is Fitting-free.  Let $A$~be a (discrete or profinite,
  respectively) group of (continuous) automorphisms of~$H$ that is
  virtually class\nbd$c$ nilpotent, for some $c \geq 0$, and satisfies
  $A \cap \Inn H = \1$.  Then the semidirect product of~$H$ by~$A$ is
  hereditarily \JNNcF.
\end{thm}

The only discrete hereditarily just infinite groups that are virtually
abelian are the infinite cyclic group and the infinite dihedral group.
The only profinite hereditarily just infinite groups that are
virtually abelian are semidirect products of the $p$\nbd adic integers
by a finite (and consequently cyclic) subgroup of its automorphism
group.  Consequently, the hypothesis that $H$~is Fitting-free in the
above theorem excludes only a small number of possibilities.
Moreover, this hypothesis on~$H$ is also necessary since the
semidirect product~$H \rtimes A$ can otherwise be virtually abelian.

\begin{proof}
  Let $H$~be a hereditarily just infinite discrete group that is
  Fitting-free and $A \leq \Aut H$ be virtually class\nbd$c$
  nilpotent with $A \cap \Inn{H} = \1$.  We shall first show that the
  semidirect product $G = H \rtimes A$ is \JNNcF\@.  We shall view
  $H$~and~$A$ as subgroups of~$G$ in the natural way.  Note that as
  $H$~is Fitting-free, it is not virtually nilpotent and therefore
  neither is~$G$.

  Let $N$~be a non-trivial normal subgroup of~$G$.  If
  $H \cap N = \1$, then $[H,N] = \1$, so
  $N \leq \Cent{G}{H} = \Centre{H}$ by use of
  Lemma~\ref{lem:noInn-centralizer}.  This is a contradiction and so
  $H \cap N \neq \1$.  Thus $H \cap N$~is of finite index in~$H$.
  Then $G/(H \cap N)$~has a copy of the group~$A$ as a subgroup of
  finite index and is therefore also virtually class\nbd$c$ nilpotent.
  We deduce that $G/N$~is virtually class\nbd$c$ nilpotent and hence
  $G$~is \JNNcF, as claimed.

  Now let $L$~be a normal subgroup of finite index in~$G$ and let
  $N$~be a non-trivial normal subgroup of~$L$.  If $H \cap N = \1$,
  then $[H \cap L, N] \leq H \cap N = \1$, so $N \leq \Cent{G}{H \cap
    L}$.  By Lemma~\ref{lem:FittingFree-Cent}, this is impossible
  since $H \cap L$~is a normal subgroup of~$G$ that is non-trivial
  (since it has finite index in~$H$) and we have already observed
  $G$~is \JNNcF\@.

  Therefore $H \cap N \neq \1$.  Since $H$~is hereditarily just
  infinite, $H \cap N$~has finite index in~$H \cap L$.  Moreover,
  $H \cap N$~is normalized by~$L$ and hence has finitely many
  conjugates in~$G$, each of which also has finite index in~$H$.  We
  deduce that $R = \Core{G}{H \cap N}$ is non-trivial, so $G/R$~is
  virtually class\nbd$c$ nilpotent.  Since $R \leq N$, we conclude
  that $L/N$~is virtually class\nbd$c$ nilpotent.

  We have shown that every normal subgroup of finite index in~$G$ is
  \JNNcF\ and therefore $G$~is hereditarily \JNNcF\ by
  Lemma~\ref{lem:HJNNF-normalredn}.

  The situation when $H$~is profinite and $A$~consists of continuous
  automorphisms of~$H$ is established by the same argument.  The only
  difference is that one needs $A$~to have the structure of a
  profinite group under the topology induced from the group~$\Autc{H}$
  of topological automorphisms of~$H$ in order that $G = H \rtimes A$
  is a profinite group.
\end{proof}

\subsection{Hereditarily \JNNcF\ groups via iterated wreath products}

We shall now construct abelian groups of automorphisms of some just
infinite groups that arise as iterated wreath products of non-abelian
finite simple groups.  We permit two possible options for the action
used for the permutational wreath product at each step.  The just
infinite groups constructed are closely related to those in Wilson's
Construction~A~\cite{JSW-Large}, though he uses two applications of
the permutational wreath product at each stage.  If one employs the
product action option~(P) at each step of our construction, then the
inverse limit constructed would be a special case of what Vannacci
terms a \emph{generalized Wilson group}
(see~\cite[Definition~3]{Vann}).  Vannacci makes use
of~\cite[Theorem~6.2]{Reid12} to determine that the profinite groups
concerned are hereditarily just infinite (and his groups also satisfy
the hypotheses of the corrected version in~\cite{Reid-corr}).  Since
we also wish to construct discrete examples of hereditarily just
infinite groups via a direct limit, we shall present a direct
verification as the discrete and profinite cases are closely linked.
This verification is somewhat general since it only requires the
action employed to be transitive and subprimitive (in the sense
of~\cite{Reid12}).  We shall then specialize to regular actions and
product actions in Example~\ref{ex:wreath} when constructing
automorphisms of the resulting hereditarily just infinite groups so as
to apply Theorem~\ref{thm:JI-automs}.

We first recall the definition of what is meant by a subprimitive
action:

\begin{defn}[\protect{\cite[Definition~1.4]{Reid12}}]
  Let $\Omega$~be a set and $H$~be a permutation group on~$\Omega$.
  We shall say that $H$~acts \emph{subprimitively} on~$\Omega$ if
  every normal subgroup~$K$ of~$H$ acts faithfully on every $K$\nbd
  orbit.
\end{defn}

Let $X_{0}$,~$X_{1}$, $X_{2}$,~\dots\ be a sequence of non-abelian
finite simple groups.  Define $W_{0} = X_{0}$.  Suppose that for
some~$n \geq 1$, we have a constructed a group~$W_{n-1}$ and choose
some faithful, transitive and subprimitive action of~$W_{n-1}$ on a
finite set~$\Omega_{n-1}$.  We define $W_{n} = X_{n}
\wr_{\Omega_{n-1}} W_{n-1}$ to be the wreath product of~$X_{n}$
by~$W_{n-1}$ and write $B_{n} = X_{n}^{\Omega_{n-1}}$ for its base
group.  We shall assume at this point that such an action always
exists, while in Example~\ref{ex:wreath} below we describe possible
examples.  Write $\rho_{n} \colon W_{n} \to W_{n-1}$ for the natural
surjective homomorphism associated to the wreath product and also note
that $W_{n-1}$~occurs as a subgroup of~$W_{n}$ so we have a chain of
inclusions: $W_{0} \leq W_{1} \leq W_{2} \leq \dots$.  We shall
write~$W$ to denote the direct limit~$\dirlim W_{n}$ of these wreath
products and $\hat{W}$~to denote the inverse limit~$\invlim W_{n}$.
It will be convenient to view~$W$ as the union of the groups~$W_{n}$.

The following is the key observation required to show that $W$~is a
hereditarily just infinite (discrete) group and $\hat{W}$~is a
hereditarily just infinite profinite group.

\begin{lemma}
  \label{lem:Wn-subnormal}
  Let $X$~be a non-abelian simple group and $H$~be a permutation group
  on a finite set~$\Omega$ that acts transitively and subprimitively.
  Define $W = X \wr_{\Omega} H$ to be the wreath product of~$X$ by~$H$
  with respect to this action and $B$~to be the base group of~$W$.
  Let $K$~be a normal subgroup of~$W$ and $N$~be a normal subgroup
  of~$K$ such that $N \nleq B$.  Then $B \leq N$.
\end{lemma}

\begin{proof}
  Write $\pi \colon W \to H$ for the natural map associated to the
  wreath product.  Since $H$~acts transitively and faithfully
  on~$\Omega$, it easily follows that $B$~is the unique minimal normal
  subgroup of~$W$.  Therefore $B \leq K$, so we may write $K = B
  \rtimes L$ where $L$~is a normal subgroup of~$H$.  Write $\Omega =
  \Gamma_{1} \cup \Gamma_{2} \cup \dots \cup \Gamma_{k}$ as the
  disjoint union of the orbits of~$L$.  Since $H$~is assumed to act
  subprimitively, $L$~acts faithfully on each~$\Gamma_{i}$.

  Since $N \nleq B$ by hypothesis, $M = N\pi$ is a non-trivial normal
  subgroup of~$L$, so the orbits of~$M$ on~$\Gamma_{i}$ form a block
  system for~$L$.  Consequently, $M$~must act without fixed points on
  each~$\Gamma_{i}$, as otherwise $M$~would fix all points
  of~$\Gamma_{i}$ and then lie in the kernel of the action of~$L$
  on~$\Gamma_{i}$.  Therefore $M$~acts without fixed points
  on~$\Omega$.  Let us write
  \[
  B = Q_{1} \times Q_{2} \times \dots \times Q_{m}
  \]
  where each $Q_{j} = X^{\Delta_{j}}$ corresponds to an
  orbit~$\Delta_{j}$ of~$M$ on~$\Omega$.  Let us suppose, for a
  contradiction, that $B \nleq N$.  Then $Q_{j} \nleq N$ for some~$j$.
  Since $M$~permutes the factors of~$Q_{j}$ transitively, $Q_{j}$~is a
  minimal normal subgroup of $BM = BN$.  However, $B \leq K$ so
  $B$~normalizes~$N$ and hence $Q_{j} \cap N$~is normal in~$BN$.  We
  deduce that $Q_{j} \cap N = \1$ and hence $[Q_{j},N] = \1$.  This
  implies that $BN$~fixes all the direct factors of~$Q_{j}$, which is
  a contradiction.  This establishes that $B \leq N$, as claimed.
\end{proof}

\begin{cor}
  \label{cor:wreath-HJI}
  \begin{enumerate}
  \item \label{i:W-HJI}
    The group $W = \dirlim W_{n}$ is hereditarily just infinite.
  \item \label{i:What-HJI}
    The profinite group $\hat{W} = \invlim W_{n}$ is hereditarily just
    infinite.
  \end{enumerate}
\end{cor}

\begin{proof}
  \ref{i:W-HJI}~Let $K$~be a normal subgroup of finite index in~$W$
  and $N$~be a non-trivial normal subgroup of~$K$.  Then $N \cap W_{k}
  \neq \1$ for some~$k$.  Consequently $N \cap W_{n} \nleq B_{n}$ for
  all $n \geq k+1$.  Applying Lemma~\ref{lem:Wn-subnormal} with~$W =
  W_{n}$, we deduce $B_{n} \leq N \cap W_{n}$ for each $n \geq k+1$.
  Hence $\langle B_{k+1}, B_{k+2}, \dots \rangle$~is contained in~$N$
  and the former is the kernel of the surjective homomorphism $W \to
  W_{k}$.  It follows that $K/N$~is finite and this shows that $W$~is
  hereditarily just infinite.

  \ref{i:What-HJI}~We shall write $\pi_{n} \colon \hat{W} \to W_{n}$
  for the surjective homomorphisms associated with the inverse limit.
  Let $K$~be an open normal subgroup of~$\hat{W}$ and $N$~be a
  non-trivial closed normal subgroup of~$K$.  Then $N\pi_{k} \neq \1$
  for some~$k$.  Now $N\pi_{n} \normal K\pi_{n} \normal W_{n}$ and
  $N\pi_{n} \nleq B_{n}$ for all $n \geq k+1$.  Hence by
  Lemma~\ref{lem:Wn-subnormal}, \ $B_{n} \leq N\pi_{n}$ for all $n
  \geq k+1$; that is, $\ker\rho_{n-1} \leq N\pi_{n}$ for all $n \geq
  k+1$.  It follows that $\ker\pi_{n-1} \leq N \ker\pi_{n}$ for all $n
  \geq k+1$.  As the kernels form a neighbourhood base for the
  identity in~$\hat{W}$, we conclude that
  \[
  \ker\pi_{k} \leq \bigcap_{n=0}^{\infty} N \ker\pi_{n} = \overline{N}
  = N.
  \]
  Since $\hat{W}/\ker\pi_{k} \cong W_{k}$ is finite, it follows that
  $K/N$~is finite.  This establishes that $\hat{W}$~is hereditarily
  just infinite.
\end{proof}

We now specify the examples of subprimitive actions that we shall use
and construct abelian groups of automorphisms of the iterated wreath
products.

\begin{example}
  \label{ex:wreath}
  As before, let $X_{0}$,~$X_{1}$, $X_{2}$, \dots\ be a sequence of
  non-abelian finite simple groups.  Define $W_{0} = X_{0}$,
  \ $\Omega_{0} = X_{0}$ also, and let $W_{0}$~act regularly
  on~$\Omega_{0}$.  We shall also define $B_{0} = W_{0}$ for use
  later.  Suppose that, for $n \geq 1$, we have constructed~$W_{n-1}$
  with a specified action on a set~$\Omega_{n-1}$.  As above, define
  $W_{n} = X_{n} \wr_{\Omega_{n-1}} W_{n-1}$ and write $B_{n} =
  X_{n}^{\, \Omega_{n-1}}$ for its base group.  There are then two
  options for the action of~$W_{n}$ on some set~$\Omega_{n}$:
  \begin{description}
  \item[(R)] Take $\Omega_{n} = W_{n}$ and let $W_{n}$~act regularly
    upon~$\Omega_{n}$; or
  \item[(P)] let $X_{n}$~act regularly on itself and use the
    \emph{product action} of~$W_{n}$ on $\Omega_{n} = B_{n} =
    X_{n}^{\, \Omega_{n-1}}$.
  \end{description}
\end{example}

For more information upon the product action of a wreath product see,
for example,~\cite[Section~2.7]{DM}.  In the case~(P) of the product
action, the elements of~$B_{n}$ act regularly on the set~$\Omega_{n}$
while the elements of~$W_{n-1}$ act to permute the factors; that is,
the action of~$W_{n-1}$ on~$\Omega_{n}$ coincides with the conjugation
action of~$W_{n-1}$ on the base group~$B_{n}$ of~$W_{n}$.  It is
immediate from the definition that the regular action of~$W_{n}$ is
subprimitive.  The product action is faithful and transitive and the
following ensures shows that it is a valid choice for our
construction.

\begin{lemma}
  Let $X$~be a non-abelian finite simple group acting regularly upon
  itself and $H$~be a transitive permutation group on a finite
  set~$\Omega$.  Then the product action of $W = X \wr_{\Omega} H$ on
  the base group $B = X^{\Omega}$ is subprimitive.
\end{lemma}

\begin{proof}
  By transitivity of~$H$ on~$\Omega$, $B$~is the unique minimal normal
  subgroup of~$W$.  Consequently, if $K$~is a normal subgroup of~$W$
  then $B \leq K$.  In the product action, $B$~acts regularly and
  hence $K$~is transitive on~$B$.  Thus, as the product action is
  faithful, it follows that the action of~$K$ on the only $K$\nbd
  orbit is also faithful.
\end{proof}

Corollary~\ref{cor:wreath-HJI} therefore applies and tells us that $W
= \dirlim W_{n}$ and $\hat{W} = \invlim W_{n}$ are hereditarily just
infinite.  We shall now construct some examples of abelian subgroups
of the automorphism groups of these groups.  There has been much study
of automorphism groups of wreath products (see, for
example,~\cite{MH}), but our requirement is simply to produce some
automorphisms that commute and so we choose not to use the full power
of such studies.

Suppose that, for each~$i \geq 0$, $\phi_{i}$~is an automorphism of
the simple group~$X_{i}$.  We take $\psi_{0} = \phi_{0}$.  Suppose
that at stage~$n-1$, we have constructed an automorphism~$\psi_{n-1}$
of~$W_{n-1}$.  Since the action of~$W_{n-1}$ on~$\Omega_{n-1}$ is
either regular or the product action (with $\Omega_{n-1} = B_{n-1}$ in
the latter case), $\psi_{n-1}$~induces a permutation of~$\Omega_{n-1}$
(that we also denote by~$\psi_{n-1}$) with the property that
\begin{equation}
  (\omega^{y})\psi_{n-1} = (\omega\psi_{n-1})^{y\psi_{n-1}}
  \label{eq:autom-act}
\end{equation}
for all $\omega \in \Omega_{n-1}$ and $y \in W_{n-1}$.  We define a
bijection $\psi_{n} \colon W_{n} \to W_{n}$ as follows:
\[
\psi_{n} \colon (x_{\omega}) y \mapsto
\bigl( (x_{\omega\psi_{n-1}^{\, -1}})\phi_{n} \bigr) (y\psi_{n-1})
\]
where $x_{\omega} \in X_{n}$ for each $\omega \in \Omega_{n-1}$ and $y
\in W_{n-1}$.  (Here we are writing elements of the base group~$B_{n}$
as sequences~$(x_{\omega})$ indexed by~$\Omega_{n-1}$ with $x_{\omega}
\in X_{n}$ in the $\omega$\nbd coordinate.)  Thus the effect
of~$\psi_{n}$ on elements in the base group is to apply~$\phi_{n}$ to
each coordinate \emph{and} permute the coordinates using the
permutation~$\psi_{n-1}$ of~$\Omega_{n-1}$, while we simply apply the
previous automorphism~$\psi_{n-1}$ to elements in the
complement~$W_{n-1}$.  It is a straightforward calculation to verify
that the resulting map is an automorphism of~$W_{n}$ and by
construction it restricts to~$\psi_{n-1}$ on the subgroup~$W_{n-1}$.
(Indeed, in the case~(R), the group~$W_{n}$ is the standard wreath
product of~$X_{n}$ by~$W_{n-1}$.  If we write $\phi = \phi_{n}$ and
$\beta = \psi_{n-1}$, then $\psi_{n} = \phi^{\ast}\beta^{\ast}$ is the
composite of the automorphisms $\phi^{\ast}$~and~$\beta^{\ast}$
introduced on pages~474 and~476, respectively, of~\cite{NeuNeu}.  The
verification for the product action case~(P) is similarly
straightforward and depends primarily on
Equation~\eqref{eq:autom-act}.)

The final result is that, for each~$n$, we have constructed an
automorphism~$\psi_{n}$ of~$W_{n}$ that extends all the previous
automorphisms.  As a consequence, we certainly have determined an
automorphism~$\psi$ of~$W$ whose restriction to each~$W_{n}$ coincides
with~$\psi_{n}$ and an automorphism~$\hat{\psi}$ of the
group~$\hat{W}$ such that $\hat{\psi} \pi_{n} = \pi_{n} \psi_{n}$ for
each~$n$ (where, as above, we write $\pi_{n} \colon \hat{W} \to W_{n}$
for the surjective homomorphism determined by the inverse limit).  The
key properties of the automorphisms that we have constructed are as
follows:

\begin{lemma}
  \label{lem:wreath-psi}
  Let $(\phi_{i}), (\phi_{i}')$ be sequences of automorphisms with
    $\phi_{i}, \phi_{i}' \in \Aut X_{i}$ for each~$i$.  Define
  $\psi$~and~$\hat{\psi}$ to be the automorphisms of\/
  $W$~and~$\hat{W}$ determined by the sequence~$(\phi_{i})$ and
  $\psi'$~and~$\hat{\psi}'$ those determined by~$(\phi_{i}')$.  Then
  \begin{enumerate}
  \item \label{i:psihat-cont}
    $\hat{\psi}$~is a continuous automorphism of\/~$\hat{W}$;
  \item \label{i:psi-comp}
    $\psi\psi'$~and~$\hat{\psi}\hat{\psi}'$ are the automorphisms of\/
    $W$~and~$\hat{W}$, respectively, determined by the
    sequence~$(\phi_{i}\phi_{i}')$;
  \item \label{i:psi-outer}
    if, for some $n \geq 0$, \ $\phi_{0}$,~$\phi_{1}$,
    \dots,~$\phi_{n-1}$ are the identity maps and $\phi_{n}$~is an
    outer automorphism of~$X_{n}$, then $\psi$~is an outer
    automorphism of\/~$W$ and $\hat{\psi}$~is an outer automorphism
    of\/~$\hat{W}$.
  \end{enumerate}
\end{lemma}

\begin{proof}
  \ref{i:psihat-cont}~By construction, $\hat{\psi}$~fixes the
  kernels~$\ker\pi_{n}$ associated to the inverse limit.  These form a
  neighbourhood base for the identity and so we deduce that
  $\hat{\psi}$~is continuous.

  \ref{i:psi-comp}~For each~$n$, write $\psi_{n}$~and~$\psi'_{n}$ for
  the automorphisms of~$W_{n}$ determined by the sequences
  $(\phi_{i})$~and~$(\phi'_{i})$, respectively.  One computes that,
  for $n \geq 1$, the composite~$\psi_{n}\psi'_{n}$ is given by
  \[
  (x_{\omega})y \mapsto \bigl(
  (x_{\omega(\psi_{n-1}')^{-1}\psi_{n-1}^{\, -1}})\phi_{n}\phi'_{n}
  \bigr) (y\psi_{n-1}\psi'_{n-1}) =
  \bigl( (x_{\omega(\psi_{n-1}\psi'_{n-1})^{-1}})\phi_{n}\phi'_{n}
  \bigr) (y\psi_{n-1}\psi'_{n-1}).
  \]
  A straightforward induction argument then shows that
  $\psi_{n}\psi'_{n}$~is the automorphism of~$W_{n}$ determined by the
  sequence~$(\phi_{i}\phi'_{i})$.  The claim appearing in the lemma
  then follows.

  \ref{i:psi-outer}~Suppose that $\phi_{0}$,~$\phi_{1}$,
  \dots,~$\phi_{n-1}$ are the identity and that $\phi_{n} \notin
  \Inn{X_{n}}$.  We claim that $\psi_{m} \notin \Inn W_{m}$ for all $m
  \geq n$.  The first of these automorphisms is given by $\bigl(
  (x_{\omega})_{\omega \in \Omega_{n-1}}  \cdot y \bigr)\psi_{n} =
  (x_{\omega}\phi_{n})_{\omega \in \Omega_{n-1}} \cdot y$ for
  $x_{\omega} \in X_{n}$ and $y \in W_{n-1}$.  Suppose that
  $\psi_{n}$~is produced by conjugating by the element~$bz$ where $b
  \in B_{n}$ and $z \in W_{n-1}$.  Note that
  $\psi_{n}$~fixes~$W_{n-1}$ and hence $b$~normalizes~$W_{n-1}$.
  Since $y^{b} = [b,y^{-1}]y$ for all $y \in W_{n-1}$, we determine
  that $b$~centralizes~$W_{n-1}$.  Therefore $z \in \Centre{W_{n-1}} =
  \1$.  We then determine that $b = (b_{\omega})_{\omega \in
    \Omega_{n-1}}$ is the constant sequence and $\phi_{n}$~coincides
  with conjugation by the element~$b_{\omega}$, contrary to
  assumption.  Hence $\psi_{n}$~is an outer automorphism of~$W_{n}$.

  Now suppose, as an induction hypothesis, that $\psi_{m} \notin
  \Inn{W_{m}}$ for some $m \geq n$.  Suppose that $\psi_{m+1}$~is
  produced by conjugating by~$bz$ where $b \in B_{m+1}$ and $z \in
  W_{m}$.  Then $b$~fixes~$W_{m}$ and hence centralizes this
  subgroup.  Consequently, $\psi_{m}$, which is the restriction
  of~$\psi_{m+1}$ to~$W_{m}$ is given by conjugating by~$z$.  This
  contradicts the inductive hypothesis.  We conclude that
  $\psi_{m}$~is an outer automorphism for all~$m \geq n$.  It now
  immediately follows that $\psi$~is an outer automorphism of~$W$ and
  $\hat{\psi}$~is an outer automorphism of~$\hat{W}$.
\end{proof}

\begin{thm}
  \label{thm:wreath-A}
  Let $X_{0}$,~$X_{1}$, \dots\ be a sequence of non-abelian finite
  simple groups.  Define $W$~to be the direct limit and $\hat{W}$~to
  be the inverse limit of the wreath products~$W_{n}$ constructed as
  in Example~\ref{ex:wreath}.  Suppose that, for each~$i \geq 0$,
  $\phi_{i}$~is an automorphism of~$X_{i}$ such that
  $\langle\phi_{i}\rangle \cap \Inn{X_{i}} = \1$.  Then the group $A =
  \prod_{i=0}^{\infty} \langle \phi_{i} \rangle$ embeds naturally
  \begin{enumerate}
  \item \label{i:W-abelianautoms}
    as a subgroup of~$\Aut W$ such that $A \cap \Inn{W} = \1$;
  \item \label{i:What-abelianautoms}
    as a profinite subgroup of~$\Autc{\hat{W}}$ such that $A \cap
    \Inn{\hat{W}} = \1$.
  \end{enumerate}
\end{thm}

Combining this theorem with Theorem~\ref{thm:JI-automs} and
Corollary~\ref{cor:wreath-HJI} produces examples of hereditarily JNAF
discrete and profinite groups.

\begin{proof}
  \ref{i:W-abelianautoms}~Each element~$g$ of $A =
  \prod_{i=1}^{\infty} \langle\phi_{i}\rangle$ is a
  sequence~$(\phi_{i}^{\, k_{i}})$ of automorphisms.  Let
  $\psi_{g}$~denote the automorphism of~$W$ determined by this
  sequence.  By Lemma~\ref{lem:wreath-psi}\ref{i:psi-comp}, the map $g
  \mapsto \psi_{g}$ is a homomorphism $\theta \colon A \to \Aut W$.
  It is clearly injective while part~\ref{i:psi-outer} of the lemma
  ensures that the image satisfies $A\theta \cap \Inn W = \1$.

  \ref{i:What-abelianautoms}~As with the first part, each~$g$ in $A =
  \prod_{i=1}^{\infty} \langle\phi_{i}\rangle$ determines a continuous
  automorphism~$\hat{\psi}_{g}$ of~$\hat{W}$.  Hence there is an
  injective homomorphism $\theta \colon A \to \Autc{\hat{W}}$ given by
  $g \mapsto \hat{\psi}_{g}$.  The subgroups $\Gamma_{n} = \set{
    \gamma \in A\theta }{[\hat{W},\gamma] \leq \ker\pi_{n}}$, for $n
  \geq 0$, form a neighbourhood base for the identity in the subspace
  topology on~$A\theta$ (see~\cite[Section~5.2]{DDMS}) and the inverse
  image of~$\Gamma_{n}$ under~$\theta$ is $\prod_{i \geq n+1}
  \langle\phi_{i}\rangle$, which is open in the product topology
  on~$A$.  Hence $\theta$~is continuous and so its image is a
  profinite subgroup of~$\Autc{\hat{W}}$ that is topologically
  isomorphic to~$A$ and satisfies $A\theta \cap \Inn \hat{W} = \1$ by
  Lemma~\ref{lem:wreath-psi}\ref{i:psi-outer}.
\end{proof}

\begin{example}
  \label{ex:JNAF1}
  As a concrete example to finish our discussion of iterated wreath
  products, fix a prime number~$p$ and let $(n_{i})$~be a sequence of
  positive integers.  Take $X_{i} = \PSL{2}{p^{n_{i}}}$, so that
  $X_{i}$~has an outer automorphism~$\phi_{i}$ of order~$n_{i}$
  induced by the Frobenius automorphism of the finite
  field~$\Field{p^{n_{i}}}$.  Then Theorem~\ref{thm:wreath-A} shows
  that the group $A = \prod_{i=0}^{\infty} C_{n_{i}}$ appears as a
  subgroup of the automorphism group of  the direct limit~$W$ with $A
  \cap \Inn{W} = \1$ and as a profinite subgroup of~$\Autc{\hat{W}}$
  with $A \cap \Inn{\hat{W}} = \1$.

  Many examples of profinite groups occur as closed subgroups of such
  a Cartesian product.  For example, by taking a suitable
  enumeration~$(n_{i})$ of prime-powers, we can embed the Cartesian
  product of countably many copies of the profinite
  completion~$\hat{\Zint}$ of the integers in some suitable
  product~$A$ and hence use Theorem~\ref{thm:JI-automs} to construct a
  hereditarily JNAF profinite group of the form
  \[
  \left( \invlim W_{n} \right) \rtimes \prod_{i=0}^{\infty} \hat{\Zint}.
  \]
\end{example}

\subsection{Hereditarily \JNNcF\ groups via Wilson's Construction~B}

The second examples of hereditarily just infinite group that we shall
consider are those introduced by Wilson~\cite{JSW-Large} in his
Construction~B.  We recall this construction here in order that we can
describe some automorphisms of these groups.  We make one notational
adjustment to Wilson's recipe: When constructing the
group~$G_{n}$, he defines $s = \order{U_{n-1}}$ and views $G_{n-1} =
U_{n-1} \rtimes L_{n-1}$ as a subgroup of the symmetric group of
degree~$s$ via its action upon~$U_{n-1}$.  Accordingly, various
elements in his construction have an integer~$i$ as a parameter with
$1 \leq i \leq s$.  In our description, we shall index using the
elements of~$U_{n-1}$ since this will aid our defining automorphisms
of the constructed groups.  We refer to~\cite{JSW-Large} for
justification of the assertions made when describing the
construction.

Let $(p_{n})$, for~$n \geq 1$, and~$(q_{n})$, for $n \geq 0$, be two
sequences of prime numbers such that, for every $n \geq 1$, $p_{n}
\neq 2$, \ $p_{n}$~divides~$q_{n}-1$ and $q_{n-1} \neq p_{n}$.  Also
let $(t_{n})$~be a sequence of positive integers.  We now describe the
construction of a sequence~$G_{n}$ of finite soluble groups.

First define $G_{0} = U_{0}$ to be the additive group of the finite
field~$\Field{q_{0}}$ and take $L_{0} = \1$.  In particular,
$G_{0}$~is cyclic of order~$q_{0}$.

Now suppose that we have constructed a group $G_{n-1} = U_{n-1}
\rtimes L_{n-1}$ where $U_{n-1}$~is the unique minimal normal subgroup
of~$G_{n-1}$ and $U_{n-1}$~is an elementary abelian $q_{n-1}$\nbd
group.  To simplify notation, write $U = U_{n-1}$ and let
$G_{n-1}$~act upon~$U$ by using the regular action of~$U_{n-1}$ upon
itself and the conjugation action of~$L_{n-1}$ upon the normal
subgroup~$U_{n-1}$.  Define
\[
\Gamma = U \times \{1,2,\dots,t_{n}\} =
\set{(u,k)}{u \in U, \ 1 \leq k \leq t_{n}}.
\]
Let $A$~be an elementary abelian $p_{n}$\nbd group with basis
$\set{a_{\gamma}}{\gamma \in \Gamma}$ and $V$~be the group
algebra~$\Field{q_{n}}A$.  Let $\zeta$~be an element of order~$p_{n}$
in the multiplicative group of the field~$\Field{q_{n}}$.  Define
invertible linear maps $x_{\delta}$,~$y_{\delta}$ (for $\delta \in
\Gamma$) and~$z$ of~$V$ by $x_{\delta} \colon v \mapsto va_{\delta}$
for~$v \in V$, \ $y_{\delta} \colon \prod  a_{\gamma}^{\, r_{\gamma}}
\mapsto \zeta^{r_{\delta}} \prod a_{\gamma}^{\,r_{\gamma}}$ for each
$\prod a_{\gamma}^{\, r_{\gamma}} \in A$, and $z \colon v \mapsto
\zeta v$ for $v \in V$.  Then define the following subgroups
of~$\GL{V}$: \ $X = \group{x_{\gamma}}{\gamma \in \Gamma}$, \ $Y =
\group{y_{\gamma}}{\gamma \in \Gamma}$ and $E = \langle X,Y \rangle$.
The action of~$G_{n-1}$ upon~$U$ induces an action on~$\Gamma$ and
hence an action on the basis of~$A$: \ $a_{(u,k)}^{\,\, g} =
a_{(u^{g},k)}$ for each $u \in U$, \ $1 \leq k \leq t_{n}$ and $g \in
G_{n-1}$.  Hence each~$g \in G_{n-1}$ determines an invertible linear
transformation of~$V$ and this normalizes both $X$~and~$Y$
(see~\cite[(4.3)]{JSW-Large}).

Now fix some element $u_{0} \in U$.  Set $\tilde{\Gamma} = \set{ (u,k)
  \in \Gamma}{ u \neq u_{0} }$ and, for $(u,k) \in \tilde{\Gamma}$,
let
\[
\tilde{a}_{(u,k)} = a_{(u_{0},k)}^{\,\, -1} a_{(u,k)}, \qquad
\tilde{x}_{(u,k)} = x_{(u_{0},k)}^{\,\, -1} x_{(u,k)}, \qquad
\tilde{y}_{(u,k)} = y_{(u_{0},k)}^{\,\, -1} y_{(u,k)}.
\]
Define $\tilde{A} = \group{\tilde{a}_{\gamma}}{\gamma \in
  \tilde{\Gamma}}$, \ $\tilde{X} = \group{\tilde{x}_{\gamma}}{\gamma
  \in \tilde{\Gamma}}$, \ $\tilde{Y} =
\group{\tilde{y}_{\gamma}}{\gamma \in \tilde{\Gamma}}$, and $D =
\langle \tilde{X}, \tilde{Y} \rangle$.  In~\cite[(4.2)]{JSW-Large}, it
is observed that
\begin{equation}
  [x_{\gamma},y_{\delta}] = \begin{cases}
    z &\text{if $\gamma = \delta$} \\
    1 &\text{if $\gamma \neq \delta$}.
  \end{cases}
  \label{eq:[x,y]}
\end{equation}
Since $z$~is central, it follows that $E$~is nilpotent of class~$2$
and that $D = \tilde{X} \tilde{Y} \langle z \rangle$.  Also set $W =
\Field{q_{n}}\tilde{A}$.  Then $W$~is an irreducible $D$\nbd module
and the group~$G_{n-1}$, via its action on~$V$, normalizes~$D$ and
induces automorphisms of~$W$ (see~\cite[(4.4) \&~(4.5)]{JSW-Large}).
Finally set $G_{n} = (W \rtimes D) \rtimes G_{n-1}$, \ $U_{n} = W$ and
$L_{n} = D \rtimes G_{n-1}$.  Associated to this semidirect product,
there are surjective homomorphisms $G_{n} \to G_{n-1}$ and inclusions
$G_{n-1} \hookrightarrow G_{n}$.  Let $\hat{G} = \invlim G_{n}$ and $G =
\dirlim G_{n}$ be the associated inverse and direct limits.

\begin{prop}[{{\cite[(4.7)]{JSW-Large}}}]
  The inverse limit~$\hat{G}$ is a hereditarily just infinite
  profinite group and the direct limit~$G$ is a hereditarily just
  infinite (discrete) group.
\end{prop}

We need the following additional properties of the groups~$G_{n}$ that
are not recorded in Wilson's paper:

\begin{lemma}
  \label{lem:JSW-centres}
  \begin{enumerate}
  \item \label{i:Z(Gn)}
    For $n \geq 1$, the centre of~$G_{n}$ is trivial.
  \item \label{i:Z(D.G0)}
    If $n = 1$ and $p_{1}$~divides~$q_{0}-1$, then the centre of~$D
    \rtimes G_{0}$ is cyclic generated by~$z$.
  \item \label{i:Z(D)}
    If $n \geq 1$ and $p_{n}$~divides~$q_{n-1} - 1$, then the centre
    of~$D$ is cyclic generated by~$z$.
  \end{enumerate}
\end{lemma}

\begin{proof}
  \ref{i:Z(Gn)}~It is observed in~\cite[(4.6)(b)]{JSW-Large} that
  $\Cent{G_{n}}{W} = W$.  Hence $\Centre{G_{n}} \leq W$.  However,
  note $z \in D'$ by~\cite[(4.4)(a)]{JSW-Large} and $w^{z} = \zeta w$
  for all $w \in W$ and so only the identity (that is, the zero vector
  in~$W$) commutes with all elements of~$G_{n}$.

  \ref{i:Z(D.G0)}~Suppose that $p_{1}$~divides~$q_{0}-1$ and recall
  that $U = G_{0}$ when $n = 0$.  Consider first the action of~$G_{0}$
  on $X = \group{x_{\gamma}}{\gamma \in \Gamma}$.  The group~$X$ is an
  elementary abelian $p_{n}$\nbd group and so as an
  $\Field{p_{n}}G_{0}$\nbd module is a direct sum $X =
  \bigoplus_{k=1}^{t_{1}} X_{k}$ where $X_{k}$~is isomorphic to the
  group algebra~$\Field{p_{n}}G_{0}$ (since $G_{0}$~acts regularly
  on~$U$ in this case).  There is a unique $1$\nbd dimensional
  submodule of~$X_{k}$ upon which $G_{0}$~acts trivially, namely that
  generated by the product~$v_{k} = \prod_{u \in U} x_{(u,k)}$ and an
  element of~$X$ is fixed by~$G_{0}$ if and only if it belongs to $P =
  \langle v_{1},v_{2},\dots,v_{t_{1}} \rangle$.

  Now if $v_{k}$~were an element of~$\tilde{X}$, it could be written
  as $v_{k} = \prod_{u \neq u_{0}} \tilde{x}_{(u,k)}^{\,\, r_{u}}$ for
  some values~$r_{u}$; that is, $v_{k} = x_{(u_{0},k)}^{\,\, -s}
  \prod_{u \neq u_{0}} x_{(u,k)}^{\,\, r_{u}}$ where $s = \sum_{u \neq
    u_{0}} r_{u}$.  Hence $r_{u} = 1$ for all~$u \neq u_{0}$ but then
  $s = \order{U}-1 \equiv 0 \pmod{p_{1}}$ since
  $p_{1}$~divides~$q_{0}-1$.  This is a contradiction and so we
  conclude $v_{k} \notin \tilde{X}$ for all~$k$.  Since the set
  of~$\tilde{x}_{\gamma}$ for $\gamma \in \tilde{\Gamma}$ forms a
  basis for~$X$, we deduce that $\tilde{X} \cap P = \1$; that is, only
  the identity element of~$\tilde{X}$ is fixed under the action
  of~$G_{0}$.  Similarly, only the identity element is fixed under the
  action of~$G_{0}$ on~$\tilde{Y}$.  From these observations, we
  deduce that if $a = ghz^{t}$ is centralized by~$G_{0}$ where $g \in
  \tilde{X}$ and $h \in \tilde{Y}$, then necessarily $g = h = 1$.  The
  claim that $\Centre{D \rtimes G_{0}} = \langle z \rangle$ then
  follows.

  \ref{i:Z(D)}~Suppose that $p_{n}$~divides~$q_{n-1}-1$.  Let $a =
  ghz^{t}$ be in the centre of~$D$ where $g \in \tilde{X}$ and $h \in
  \tilde{Y}$.  From Equation~\eqref{eq:[x,y]}, it follows that, for
  $\gamma = (u,k)$ and $\delta = (v,\ell)$ with $u,v \neq u_{0}$,
  \[
  [\tilde{x}_{\gamma},\tilde{y}_{\delta}]
  = [x_{(u_{0},k)}^{\,\, -1}x_{(u,k)}, y_{(u_{0},\ell)}^{\,\,
      -1}y_{(v,\ell)}]
  = \begin{cases}
    z^{2} &\text{if $\gamma = \delta$}, \\
    z &\text{if $k = \ell$ and $u \neq v$}, \\
    1 &\text{if $k \neq \ell$}.
  \end{cases}
  \]
  Suppose $g = \prod_{\gamma \in \tilde{\Gamma}}
  \tilde{x}_{\gamma}^{\, r_{\gamma}}$.  Then, for $\delta = (v,\ell)
  \in \tilde{\Gamma}$, \ $[g,\tilde{y}_{\delta}] =
  z^{N_{\ell}+r_{\delta}}$ where $N_{\ell} = \sum_{u \neq u_{0}}
  r_{(u,\ell)}$.  It follows that $r_{\delta} \equiv -{N_{\ell}}
  \pmod{p_{n}}$ for all $\delta = (u,\ell) \in \tilde{\Gamma}$.  Hence
  $N_{\ell} \equiv -\bigl(\order{U} - 1\bigr) N_{\ell} \equiv 0
  \pmod{p_{n}}$ for $1 \leq \ell \leq t_{n}$, using the fact that
  $U$~is an elementary abelian $q_{n-1}$\nbd group and
  $p_{n}$~divides~$q_{n-1}-1$.  This shows $r_{\delta} \equiv 0
  \pmod{p_{n}}$ for all $\delta \in \tilde{\Gamma}$ and hence $g =
  1$.  It similarly follows that $h = 1$.  We conclude that $a =
  z^{t}$ for some~$t$ and this establishes that $\Centre{D} = \langle
  z \rangle$.
\end{proof}

We shall now describe a method to construct some automorphisms of the
groups $G$~and~$\hat{G}$.  For each~$i \geq 0$, let $\lambda_{i}$~be a
non-zero scalar in the field~$\Field{q_{i}}$.  In particular,
$\psi_{0} \colon x \mapsto \lambda_{0}x$ is an automorphism of the
additive group $G_{0} = \Field{q_{0}}$.  Now suppose that we have
constructed an automorphism~$\psi_{n-1}$ of~$G_{n-1}$.  Since
$U_{n-1}$~is the unique minimal normal subgroup of~$G_{n-1}$,
$\psi_{n-1}$~induces an automorphism of~$U = U_{n-1}$.  Hence we
induce a bijection $\psi_{n-1} \colon \Gamma \to \Gamma$ by
$(u,k)\psi_{n-1} = (u\psi_{n-1},k)$ and consequently determine
an automorphism of~$A$ by $a_{\gamma} \mapsto a_{\gamma\psi_{n-1}}$
and this extends to an invertible linear map $\psi_{n-1} \colon V \to
V$.

\begin{lemma}
  \label{lem:JSW-B-conj}
  The induced linear map~$\psi_{n-1} \in \GL{V}$ satisfies
  $\psi_{n-1}^{\, -1} x_{\delta}\psi_{n-1} = x_{\delta\psi_{n-1}}$ and
  $\psi_{n-1}^{\, -1} y_{\delta}\psi_{n-1} = y_{\delta\psi_{n-1}}$ for
  each $\delta \in \Gamma$.
\end{lemma}

\begin{proof}
  If $v \in V$, then $v \psi_{n-1}^{\, -1} x_{\delta}\psi_{n-1} = (v
  \psi_{n-1}^{\, -1} \cdot a_{\delta})\psi_{n-1} = v\cdot
  a_{\delta\psi_{n-1}}$ and hence $\psi_{n-1}^{\, -1}
  x_{\delta}\psi_{n-1} = x_{\delta\psi_{n-1}}$.  For an element $\prod
  a_{\gamma}^{\, r_{\gamma}} \in A$, we compute
  \begin{align*}
    \left( \prod a_{\gamma}^{\, r_{\gamma}} \right)
    \psi_{n-1}^{\, -1} y_{\delta}\psi_{n-1}
    &= \left( \prod a_{\gamma\psi_{n-1}^{\, -1}}^{\, r_{\gamma}} \right)
    y_{\delta}\psi_{n-1} \\
    &= \left( \zeta^{r_{\delta\psi_{n-1}}} \prod
    a_{\gamma\psi_{n-1}^{\, -1}}^{\, r_{\gamma}} \right)\psi_{n-1} =
    \zeta^{r_{\delta\psi_{n-1}}} \prod a_{\gamma}^{\, r_{\gamma}}
  \end{align*}
  and hence $\psi_{n-1}^{\, -1} y_{\delta}\psi_{n-1} =
  y_{\delta\psi_{n-1}}$.
\end{proof}

As a consequence, we determine an automorphism~$\psi_{n-1}^{\ast}$ of
the subgroup~$E$ of~$\GL{V}$ given by conjugating by this linear
map~$\psi_{n-1}$.  Notice furthermore that $D\psi_{n-1}^{\ast} = D$
since
\[
\tilde{x}_{(u,k)}\psi_{n-1}^{\ast} = \psi_{n-1}^{\, -1}
x_{(u_{0},k)}^{\,\, -1} x_{(u,k)} \psi_{n-1} =
x_{(u_{0}\psi_{n-1},k)}^{\,\, -1} x_{(u\psi_{n-1},k)} =
\tilde{x}_{(u_{0}\psi_{n-1},k)}^{\,\, -1}
\tilde{x}_{(u\psi_{n-1},k)}
\]
and similary for~$\tilde{y}_{(u,k)}$.  Finally, we determine a
bijection $\psi_{n} \colon G_{n} \to G_{n}$ by
applying~$\psi_{n-1}^{\ast}$ to elements in~$D$ and
applying~$\psi_{n-1}$ to those in~$G_{n-1}$ and defining its effect on
elements of~$W = \Field{q_{n}}\tilde{A}$ by
\[
\tilde{a}_{(u,k)}\psi_{n} =
\lambda_{n} a_{(u_{0}\psi_{n-1},k)}^{\,\, -1} a_{(u\psi_{n-1},k)} =
\lambda_{n} \tilde{a}_{(u_{0}\psi_{n-1},k)}^{\,\, -1}
\tilde{a}_{(u\psi_{n-1},k)}
\]
and extending by linearity.  Thus, the effect of~$\psi_{n}$ on~$W$ is
the composite of the linear map~$\psi_{n-1}$ defined above together
with scalar multiplication by~$\lambda_{n}$.  Since each
$x_{\delta}$~and~$y_{\delta}$ is a linear map, it follows that
$\psi_{n}$~induces an automorphism of~$W \rtimes D$.  Also notice
that, since the action of~$G_{n-1}$ on~$U = U_{n-1}$ is given by the
regular action of~$U_{n-1}$ on itself and the conjugation action
of~$L_{n-1}$ on~$U_{n-1}$, the automorphism~$\psi_{n-1}$ satisfies
\[
(u^{g})\psi_{n-1} = (u\psi_{n-1})^{g\psi_{n-1}}
\]
for all $u \in U$ and $g \in G_{n-1}$ (and here exponentiation denotes
the action).  One determines, using Lemma~\ref{lem:JSW-B-conj}, that
$(x_{\delta}^{\, g}) \psi_{n-1}^{\ast} =
(x_{\delta}\psi_{n-1}^{\ast})^{g\psi_{n-1}}$ for $\delta \in \Gamma$
and $g \in G_{n-1}$.  Similar formulae hold when we conjugate
$y_{\delta}$~and~$a_{\delta}$ by elements of~$G_{n-1}$ (in the latter
case, we rely upon the fact that an element of~$G_{n-1}$ induces a
linear map on~$V$ and so commutes with the operation of multiplying by
the scalar~$\lambda_{n}$).  We conclude that $\psi_{n}$~is indeed an
automorphism of~$G_{n}$ that restricts to the previous
one~$\psi_{n-1}$ on~$G_{n-1}$.  As a consequence, we determine an
automorphism~$\psi$ of~$G$ whose restriction to each~$G_{n}$
equals~$\psi_{n}$ and an automorphism~$\hat{\psi}$ of~$\hat{G}$ such
that $\hat{\psi}\pi_{n} = \pi_{n}\psi_{n}$ for each~$n$ (where
$\pi_{n} \colon \hat{G} \to G_{n}$ is the surjective homomorphism
associated to the inverse limit).  The properties of this construction
are analogous to those for the iterated wreath product and the first
two parts of the following are established similarly to those of
Lemma~\ref{lem:wreath-psi}.

\begin{lemma}
  \label{lem:JSW-psi}
  Let $(\lambda_{i}), (\mu_{i})$ be sequences of scalars with
  $\lambda_{i},\mu_{i} \in \Field{q_{i}}^{\ast}$.  Define
  $\psi$~and~$\hat{\psi}$ to be the automorphisms of $G$~and~$\hat{G}$
  determined by the sequence~$(\lambda_{i})$ and
  $\theta$~and~$\hat{\theta}$ those determined by~$(\mu_{i})$.  Then
  \begin{enumerate}
  \item $\hat{\psi}$~is a continuous automorphism of~$\hat{G}$;
  \item $\psi\theta$~and~$\hat{\psi}\hat{\theta}$ are the
    automorphisms of $G$~and~$\hat{G}$, respectively, determined by
    the sequence~$(\lambda_{i}\mu_{i})$;
  \item \label{i:JSW-outer}
    if $p_{i}$~divides~$q_{i-1}-1$ for all $i \geq 1$ and, for some $n
    \geq 0$, $\lambda_{i} = 1$ in~$\Field{q_{i}}$ for $0 \leq i \leq
    n-1$ and $\lambda_{n}$~is not in the subgroup of order~$p_{n}$ in
    the multiplicative group of the field\/~$\Field{q_{n}}$, then
    $\psi$~is an outer automorphism of~$G$ and $\hat{\psi}$~is an
    outer automorphism of~$\hat{G}$.
  \end{enumerate}
\end{lemma}

\begin{proof}
  We prove part~\ref{i:JSW-outer}.  Suppose that
  $p_{i}$~divides~$q_{i-1}-1$ for all~$i \geq 1$ in addition to the
  original assumptions on the $p_{i}$~and~$q_{j}$.  Suppose that
  $\lambda_{i} = 1$ for $0 \leq i \leq n-1$ and that $\lambda_{n}$~is
  not a power of~$\zeta$ where $\zeta$~is an element of order~$p_{n}$
  in~$\Field{q_{n}}\mult$.  Since $\lambda_{i} = 1$ for $0 \leq i \leq
  n-1$, the automorphism~$\psi_{n-1}$ of~$G_{n-1}$ is the identity
  map.  We shall first show that $\psi_{n} \notin \Inn{G_{n}}$.  We
  will need a different argument according to the value of~$n$.  If $n
  = 0$, then $G_{0}$~is abelian so $\psi_{0}$~cannot be an inner
  automorphism as it is not the identity.

  Suppose that $n = 1$ and that $\psi_{1}$~is produced by conjugating
  by the element~$wdh$ where $w \in W$, $d \in D$ and $h \in G_{0}$.
  In this case, $\psi_{0}$~is the identity, so $\psi_{1}$~induces the
  identity on~$D \rtimes G_{0}$ and hence $dh \in \Centre{D \rtimes
    G_{0}}$; that is, $h = 1$ and $d = z^{k}$ for some~$k$ by
  Lemma~\ref{lem:JSW-centres}\ref{i:Z(D.G0)}.  Now observe that
  $w$~must normalize~$D$ since $D\psi_{1} = D$ and it follows that
  $[w,g] = 1$ for all $g \in D$.  Hence $\Field{q_{n}}w$~is a $D$\nbd
  invariant subspace of~$W$ so $w = 0$ as $W$~is an irreducible
  $D$\nbd module (by~\cite[(4.5)(c)]{JSW-Large}).  In conclusion,
  $\psi_{1}$~is the inner automorphism of~$G_{1}$ determined by
  conjugation by~$z^{k}$.  This means that $\lambda_{1} = \zeta^{k}$,
  contrary to our assumption.
  
  Suppose that $n \geq 2$ and that $\psi_{n} \in \Inn{G_{n}}$, say
  that conjugation by the element~$wdg$ (where $w \in W$, $d \in D$
  and $g \in G_{n-1}$) achieves the same effect as
  applying~$\psi_{n}$.  In particular, $wdg$~centralizes~$G_{n-1}$ and
  so $wd$~normalizes~$G_{n-1}$.  It follows that $[wd,y] = 1$ for all
  $y \in G_{n-1}$ and hence $w$~and~$d$ are both centralized
  by~$G_{n-1}$ and $g \in \Centre{G_{n-1}}$.  Therefore $g = 1$ by
  Lemma~\ref{lem:JSW-centres}\ref{i:Z(Gn)}.  Also necessarily $d \in
  \Centre{D}$, so $d = z^{k}$ for some~$k$ by
  Lemma~\ref{lem:JSW-centres}\ref{i:Z(D)}, while $w$~spans a $D$\nbd
  submodule of~$W$ and hence $w = 0$.  We conclude, as in the previous
  case, that $\psi_{n}$~is the inner automorphism of~$G_{n}$
  determined by conjugation by~$z^{k}$, which is impossible as
  $\lambda_{n} \notin \langle \zeta \rangle$ by assumption.

  Now suppose that $\psi_{m} \notin \Inn{G_{m}}$ for some $m \geq n$.
  If it were the case that $\psi_{m+1}$~is produced by conjugating
  by~$wdg$ where $w \in W$, $d \in D$ and $g \in G_{m}$, then
  $\psi_{m}$~would coincide with conjugation by~$g$, contrary to
  assumption.  Hence $\psi_{m} \notin \Inn{G_{m}}$ for all $m \geq n$.
  It now follows that $\psi$~is an outer automorphism of~$W$ and
  $\hat{\psi}$~is an outer automorphism of~$\hat{W}$.
\end{proof}

\begin{thm}
  \label{thm:JSW-Aut}
  Let $(p_{n})$, for $n \geq 1$, and $(q_{n})$, for $n \geq 0$, be a
  sequence of prime numbers such that for every $n \geq 1$, $p_{n}
  \neq 2$, $p_{n}$~divides both $q_{n-1}-1$ and~$q_{n}-1$.  Let
  $(t_{n})$~be any sequence of positive integers and define $G$~to be
  the direct limit and $\hat{G}$~to be the inverse limit of the
  semidirect products~$G_{n}$ built via Wilson's Construction~B.  Take
  $r_{0} = q_{0}-1$ and, for each~$i \geq 1$, write $q_{i}-1 = r_{i}
  p_{i}^{\, m_{i}}$ where $p_{i} \nmid r_{i}$ and let
  $C_{r_{i}}$~denote a cyclic group of order~$r_{i}$.  Then the group
  $A = \prod_{i=0}^{\infty} C_{r_{i}}$ embeds naturally
  \begin{enumerate}
  \item \label{i:JSW-abelianaut}
    as a subgroup of~$\Aut{G}$ such that $A \cap \Inn{G} = \1$;
  \item as a profinite subgroup of~$\Autc{\hat{G}}$ such that $A \cap
    \Inn{\hat{G}} = \1$.
  \end{enumerate}
\end{thm}

\begin{proof}
  The proof is similar to that of Theorem~\ref{thm:wreath-A}.  For
  each~$i$, let $\lambda_{i}$~be an element of order~$r_{i}$ in the
  multiplicative group~$\Field{q_{i}}^{\ast}$.  Then, for $i \geq 1$,
  $\langle \lambda_{i} \rangle \cap \langle \zeta_{i} \rangle = \1$
  where $\zeta_{i}$~denotes an element of order~$p_{i}$
  in~$\Field{q_{i}}\mult$.  Now if $g = (\lambda_{i}^{\, k_{i}}) \in
  \prod_{i=0}^{\infty} \langle\lambda_{i}\rangle \cong A$, write
  $\psi_{g}$~for the automorphism~$\psi$ determined by the
  sequence~$(\lambda_{i}^{\, k_{i}})$ as above.
  Lemma~\ref{lem:JSW-psi} ensures that $g \mapsto \psi_{g}$ is a
  homomorphism into~$\Aut G$ whose image satisfies the conclusion
  of~\ref{i:JSW-abelianaut}.  The second part is established
  similarly: we determine an injective homomorphism $\theta \colon
  \prod_{i=0}^{\infty} \langle\lambda_{i}\rangle \to \Autc{\hat{G}}$
  and this is continuous since the inverse image under~$\theta$ of the
  basic neighbourhood of the identity comprising those automorphisms
  that act trivially on~$G_{n}$ is $\prod_{i \geq n+1}
  \langle\lambda_{i}\rangle$.
\end{proof}

\begin{example}
  \label{ex:JNAF2}
  A specific example can be constructed as follows.  Let $(n_{i})$~be
  any sequence of positive integers.  Let $(p_{i})$, for $i \geq 1$,
  be any sequence of odd primes such that $p_{i}$~does not
  divide~$n_{i}$.  When $i \geq 1$, take $a_{i} =
  \lcm(p_{i}n_{i},p_{i+1})$ and $a_{0} = \lcm(n_{0},p_{1})$.  Now
  take, for~$i \geq 0$, $q_{i}$~to be any prime number of the
  form~$a_{i}k+1$ for some $k \in \Nat$.  (The existence of such a
  prime number is guaranteed by Dirichlet's Theorem.)  These choices
  of sequences then fulfil the requirements of
  Theorem~\ref{thm:JSW-Aut} and the integer~$r_{i}$ appearing in the
  statement is divisible by~$n_{i}$ by construction.  Consequently, we
  deduce that the Cartesian product $\prod_{i=0}^{\infty} C_{n_{i}}$
  embeds in the subgroup~$A$ appearing in Theorem~\ref{thm:JSW-Aut}.
  We may use any closed subgroup of this Cartesian product as the
  choice of~$A$ in Theorem~\ref{thm:JI-automs}.  In particular, there
  are many choices of abelian profinite groups~$A$ such that $\hat{G}
  \rtimes A$ is hereditarily JNAF including, as with the iterated
  wreath product, a hereditarily JNAF example of the form
  \[
  \left( \invlim G_{n} \right) \rtimes \prod_{i=0}^{\infty} \hat{\Zint}.
  \]
\end{example}


\subsection{Hereditarily \JNNcF\ groups by use of the Nottingham
  group}

The following construction brings together two facets of the study of
pro\nbd$p$ groups.  As a first ingredient, we make use of the work of
Lubotzky--Shalev~\cite{LS} on $R$\nbd analytic groups, in the specific
case when $R$~is the formal powers series ring $\psring$ to identify a
specific hereditarily just infinite pro\nbd$p$ group~$G$.  Secondly,
we use the fact that every countably based pro\nbd$p$ group embeds in
the the automorphism group $\Aut(R)$ to obtain a wide range of groups
of automorphisms of our group~$G$.

\begin{example}
  \label{ex:Nott}
  Let $p$~be a prime number and let $n$~be a positive integer with
  $n \geq 2$ such that $p$~does not divide~$n$.  Take $R = \psring$,
  the pro\nbd$p$ ring of all formal power series over the field of
  $p$~elements, which is a local ring with unique maximal ideal
  $\mathfrak{m} = T \, \psring$ generated by the indeterminate~$T$.
  Then take $G = \SLc{n}{R}$, the first principal congruence subgroup
  of the special linear group of all $n \times n$~matrices of
  determinant~$1$ over $R$; that is,
  \[
    G = \set{g \in \SL{n}{R}}{g \equiv I \pmod{\mathfrak{m}}},
  \]
  where $I$ denotes the $n \times n$ identity matrix.  Using the
  technology of~\cite{LS}, it is straightforward to observe that
  $G$~is a hereditarily just infinite pro\nbd$p$ group.  First $G$~is
  $R$\nbd perfect and so the terms of its lower central series are the
  congruence subgroups
  \[
  \gamma_{k}(G) = G_{k} =
  \set{g \in \SL{n}{R}}{g \equiv I \pmod{\mathfrak{m}^{k}}},
  \]
  for each~$k \geq 1$.  Adapting slightly the notation used
  in~\cite{LS}, we see that the (completed) graded Lie ring
  associated to the lower central series of $G$ satisfies
  \begin{equation*}
    L(G) = L_G(G) = \overline{ \bigoplus_{i=1}^{\infty} G_{i}/G_{i+1}
    } \cong \prod_{i=1}^{\infty} T^{i} 
    \slLie{n}{\Field{p}} \cong \slLie{n}{\mathfrak{m}},
  \end{equation*}
  the latter being the Lie algebra over~$\Field{p}$ of $n \times n$
  matrices with entries in~$\mathfrak{m}$ and trace~$0$.  To every
  closed subgroup $H$ of~$G$ we associate a closed Lie subalgebra
  of~$L(G)$ that we denote by~$L_{G}(H)$ and whose properties are
  described in~\cite[Lemma~2.13]{LS}.  Using the isomorphism above we
  view~$L_{G}(H)$ as a Lie subalgebra of $\slLie{n}{\mathfrak{m}}$.
  In particular, $L_{G}(G_{k})$~corresponds to the Lie subalgebra
  $\prod_{i=k}^{\infty} T^{i} \slLie{n}{\Field{p}} \cong
  \slLie{n}{\mathfrak{m}^{k}}$.  If $W$~is a non-zero $\Field{p}$\nbd
  subspace of~$\slLie{n}{R}$ satisfying $[W,
    \slLie{n}{\mathfrak{m}^{k}}]_{\text{Lie}} \subseteq W$ for some~$k
  \geq 1$, then a direct computation shows there exists~$r$ such that
  $\slLie{n}{\mathfrak{m}^{r}} \subseteq W$.  (It is this computation
  that uses the fact that $p \nmid n$.)

  Now let $H$~and~$N$ be closed subgroups of~$G$ such that $\1 \neq N
  \normal H \normal G$.  Then $L_{G}(H)$~is an ideal of the Lie
  algebra~$\slLie{n}{\mathfrak{m}}$ and hence there exists~$r$ such
  that $\slLie{n}{\mathfrak{m}^{r}} \subseteq L_{G}(H)$.
  Consequently, $G_{r} \leq H$, so that $[N,G_{r}] \leq N$ and one
  deduces $[L_{G}(N),\slLie{n}{\mathfrak{m}^{r}}]_{\text{Lie}}
  \subseteq L_{G}(N)$.  It follows that $\slLie{n}{\mathfrak{m}^{s}}
  \subseteq L_{G}(N)$ for some~$s$ and hence $G_{s} \leq N$ and so
  $\order{G:N} < \infty$.  This demonstrates that $G$~is hereditarily
  just infinite.

  Next we exploit properties of the Nottingham group~$\Nott$ over
  $\Field{p}$ to produce groups of automorphisms of the above
  group~$G$.  The group $\Nott$ is the Sylow pro-$p$ subgroup of the
  profinite group $\Autc(R) = \Aut(R)$; it coincides with the
  group~$\Aut^{1}(R)$ of all automorphisms of the ring~$R$ that act
  trivially modulo~$\mathfrak{m}^{2}$.  Any element $\alpha$
  of~$\Nott$ is then uniquely determined by its effect upon the
  indeterminate~$T$ and, conversely, for any $f \in R$ with
  $f \equiv T \pmod{\mathfrak{m}^{2}}$ there is a unique element of
  $\Nott$ mapping $T$ to~$f$.  (Thus~$\Nott$ could alternatively be
  defined as a group of power series $T + \mathfrak{m}^{2}$ with the
  binary operation given by substitution of power series.  For our
  construction, however, the behaviour as automorphisms is more
  relevant.)  We refer to~\cite{Camina-survey} for background material
  concerning the Nottingham group, which plays a role also in number
  theory and dynamics.

  The action of the Nottingham group~$\Nott$ on~$R$ induces a faithful
  action upon the group~$G = \SLc{n}{R}$ and hence we
  construct a subgroup~$\dot{\Nott} \leq \Autc G$ isomorphic
  to~$\Nott$.  Suppose $\alpha \in \Nott$ is an element that induces
  an inner automorphism~$\dot{\alpha}$ of the group~$G$, and put $f =
  T \alpha \in T + \mathfrak{m}^{2}$.  Then there exists a matrix~$h
  \in G$ such that $hx^{\dot{\alpha}} = xh$ for all $x \in G$.  In
  particular, upon taking $x = I + Te_{ij}$ for $1 \leq i,j \leq n$
  with $i \neq j$, we conclude that $h$~must be a diagonal matrix such
  that every pair of distinct diagonal entries $a$~and~$b$ are linked
  by the relation $Ta = fb$ in~$R$.  It follows that $f^{2} = T^{2}$
  and hence, since $f \equiv T \pmod{\mathfrak{m}^{2}}$, that $f = T$
  and $\dot{\alpha} = \mathrm{id}_G$.  In conclusion, the copy of the
  Nottingham group in~$\Autc G$ satisfies $\dot{\Nott} \cap \Inn G =
  \1$.

  As the final step in our construction, we use the result of
  Camina~\cite{Camina-embed} that every countably-based pro\nbd$p$
  group can be embedded as a closed subgroup in~$\Nott$.  Hence if
  $A$~is any finitely generated pro\nbd$p$ group that is virtually
  nilpotent (say, of class~$c$), then it may be embedded in~$\Autc G$
  in such a way that $A \cap \Inn G = \1$.  Hence we have satisfied
  the conditions of Theorem~\ref{thm:JI-automs} and the semidirect
  product~$G \rtimes A$ is an example of a hereditarily
  \JNNcF\ pro\nbd$p$ group.
\end{example}


\paragraph{Acknowledgements} Some of this research was conducted while
the second author was visiting Heinrich Heine University
D\"{u}sseldorf.  He thanks the university for its hospitality and we
gratefully acknowledge partial support by the Humboldt Foundation.  We
also thank the referees for their careful reading of the paper and
insightful suggestions, particularly that of using subprimitivity in
Section~6.


\noindent
Benjamin Klopsch, \texttt{klopsch@math.uni-duesseldorf.de}\\
Mathematisches Institut, Heinrich-Heine-Universit\"{a}t, 40225
D\"{u}sseldorf, Germany\\[5pt]
Martyn Quick, \texttt{mq3@st-andrews.ac.uk}\\
School of Mathematics \& Statistics, University of St Andrews, St
Andrews, UK


\begin{thebibliography}{99}

\bibitem{Anderson}
  M.~P. Anderson, ``Subgroups of finite index in profinite groups,''
  \textit{Pacific J. Math.}~\textbf{62} (1976), no.~1, 19--28.

\bibitem{BGS}
  L. Bartholdi, R.~I. Grigorchuk \& Z. \v{S}uni\'{k}, ``Branch
  groups,'' \textit{Handbook of Algebra}~\textbf{3}, 989--1112,
  Elsevier/North-Holland, Amsterdam, 2003.

\bibitem{BMS}
  H. Bass, J. Milnor \& J.-P. Serre, ``Solution of the congruence
  subgroup problem for $\mathrm{SL}_{n}$ ($n \geq 3$) and
  $\mathrm{Sp}_{2n}$ ($n \geq 2$), \textit{Inst.\ Hautes \'{E}tudes
    Sci.\ Publ.\ Math.}~\textbf{33} (1967), 59--137.

\bibitem{Camina-embed}
  R.~D. Camina, ``Subgroups of the Nottingham group,''
  \textit{J. Algebra}~\textbf{196} (1997), no,~1, 101--113.

\bibitem{Camina-survey}
  R.~D. Camina, ``The Nottingham group,'' \textit{New Horizons in
    Pro\nbd$p$ Groups}, 205--221, Progr.\ Math.~\textbf{184},
  Birkh\"{a}user Boston, 2000.

\bibitem{deFalco}
  M. De Falco, ``Groups whose proper quotients are
  nilpotent-by-finite,''
  \textit{Math.\ Proc.\ R. Ir.\ Acad.}~\textbf{102A} (2002), no.~2,
  131--139.

\bibitem{DMS}
  E. Detomi, M. Morigi \& P. Shumyatsky, ``Profinite groups with
  restricted centralizers of commutators,''
  \textit{Proc.\ Roy.\ Soc.\ Edinburgh Sect.~A}~\textbf{150} (2020),
  no.~5, 2301--2321.

\bibitem{Diestel}
  R.~Diestel, \textit{Graph Theory, 5th Edition}, Graduate Texts
  Math.~\textbf{173}, Springer, Berlin 2017.

\bibitem{DDMS}
  J.~D. Dixon, M.~P.~F. du Sautoy, A.~Mann \& D.~Segal,
  \textit{Analytic Pro-$p$ Groups, 2nd Edition}, Cambridge Studies
  Adv.\ Math.~\textbf{61}, CUP, Cambridge 1999.

\bibitem{DM}
  J.~D. Dixon \& B. Mortimer, \textit{Permutation Groups}, Graduate
  Texts Math.~\textbf{163}, Springer, New York 1996.

\bibitem{Fink}
  E. Fink, ``A finitely generated branch group of exponential growth
  without free subgroups,'' \textit{J. Algebra}~\textbf{397} (2014),
  625--642.

\bibitem{Grig84}
  R.~I. Grigorchuk, ``Degrees of growth of finitely generated groups
  and the theory of invariant means'' (Russian),
  \textit{Izv.\ Akad.\ Nauk SSSR Ser.\ Mat.}~\textbf{48} (1984),
  no.~5, 939--985.

\bibitem{Grig00}
  R.~I. Grigorchuk, ``Just infinite branch groups,'' \textit{New
    Horizons in Pro\nbd$p$ Groups}, 121--179,
  Progr.\ Math.~\textbf{184}, Birkh\"{a}user Boston, 2000.

\bibitem{Hardy-PhD}
  P.~D. Hardy, ``Aspects of abstract and profinite group theory,''
  Ph.D. thesis (University of Birmingham), 2002.

\bibitem{Hegedus}
  P. Heged\H{u}s, ``The Nottingham group for $p = 2$,''
  \textit{J. Algebra}~\textbf{246} (2001), no.~1, 55--69.

\bibitem{Klopsch-Nott}
  B. Klopsch, ``Normal subgroups in substitution groups of formal
  power series,'' \textit{J. Algebra}~\textbf{228} (2000), no.~1,
  91--106.

\bibitem{KOS}
  L. Kurdochenko, J. Otal \& I. Subbotin, \textit{Groups with
    prescribed quotient groups and associated module theory}, Series
  in Algebra~\textbf{8}, World Scientific Publishing Co., River Edge,
  2002.

\bibitem{LGM}
  C.~R. Leedham-Green \& S. McKay, \textit{The Structure of Groups of
    Prime Power Order}, London Math.\ Soc.\ Monographs, New Series,
  \textbf{27}, OUP, Oxford 2002.

\bibitem{LS}
  A. Lubotzky \& A. Shalev, ``On some $\Lambda$\nbd analytic
  pro\nbd$p$ groups,'' \textit{Israel J. Math.}~\textbf{85} (1994),
  307--337.

\bibitem{McCarthy1}
  D. McCarthy, ``Infinite groups whose proper quotients are
  finite,~I,'' \textit{Comm.\ Pure Appl.\ Math.}~\textbf{21} (1968),
  545--562.

\bibitem{McCarthy2}
  D. McCarthy, ``Infinite groups whose proper quotients are
  finite,~II,'' \textit{Comm.\ Pure Appl.\ Math.}~\textbf{23} (1970),
  767--789.

\bibitem{MH}
  A. Mohammadi Hassanabadi, ``Automorphisms of permutational wreath
  products,'' \textit{J. Austral.\ Math. Soc.\ (Series~A)}~\textbf{26}
  (1978), no.~2, 198--208.

\bibitem{NeuNeu}
  B.~H. Neumann \& H. Neumann, ``Embedding theorems for groups,''
  \textit{J. London Math.\ Soc.}~\textbf{34} (1959), 465--479.



\bibitem{NS1}
  N. Nikolov \& D. Segal, ``On finitely generated profinite groups,~I:
  strong completeness and uniform bounds,'' \textit{Ann.\ of
    Math.~(2)}~\textbf{165} (2007), no.~1, 171--238.

\bibitem{MRQ}
  M. Quick, ``Groups with virtually abelian proper quotients,''
  \textit{J. Lond.\ Math.\ Soc.~(2)} \textbf{75} (2007), no.~3,
  597--609.

\bibitem{Reid10a}
  C.~D. Reid, ``Subgroups of finite index and the just infinite
  property,'' \textit{J. Algebra}~\textbf{324} (2010), no.~9,
  2219--2222.

\bibitem{Reid10b}
  C.~D. Reid, ``On the structure of just infinite profinite groups,''
  \textit{J. Algebra}~\textbf{324} (2010), no.~9, 2249--2261.

\bibitem{Reid12}
  C.~D. Reid, ``Inverse system characterizations of the (hereditarily)
  just infinite property in profinite groups,'' \textit{Bull.\ London
    Math.\ Soc.}~\textbf{44} (2012), no.~3, 413--425.

\bibitem{Reid-corr}
  C.~D. Reid, ``Inverse system characterizations of the (hereditarily)
  just infinite property in profinite groups'' (corrected version
  of~\cite{Reid12}), Oct 2018. \texttt{arXiv:1708.80301v1}

\bibitem{RZ}
  L. Ribes \& P.~A. Zalesskii, \textit{Profinite Groups},
  Ergeb.\ Math.\ Grenzgeb.~(3) \textbf{40}, Springer, Berlin 2000.

\bibitem{Robinson}
  D.~J.~S. Robinson, \textit{A Course in the Theory of Groups, Second
    Edition}, Graduate Texts Math.~\textbf{80}, Springer, New York
  1996.




\bibitem{Vann}
  M. Vannacci, ``On hereditarily just infinite profinite groups
  obtained via iterated wreath products,'' \textit{J. Group
    Theory}~\textbf{19} (2016), no.~2, 233-238.

\bibitem{JSW71}
  J.~S. Wilson, ``Groups with every proper quotient finite,''
  \textit{Proc.\ Cambridge Philos.\ Soc.}~\textbf{69} (1971),
  373--391.

\bibitem{JSW00}
  J.~S. Wilson, ``On just infinite abstract and profinite groups,''
  \textit{New Horizons in pro\nbd$p$ Groups}, 181--203, Progress
  Math.~\textbf{184}, Birkh\"{a}user, Boston 2000.

\bibitem{JSW-Large}
  J.~S. Wilson, ``Large hereditarily just infinite groups,''
  \textit{J. Algebra}~\textbf{324} (2010), no.~2, 248--255.

\end{thebibliography}
\end{document}